\documentclass[reqno]{amsproc}
\usepackage{amsmath}
\usepackage{amsfonts}
\usepackage{mathrsfs}
\usepackage{amssymb}
\usepackage{circuitikz}
\usepackage{color}
\usepackage{tikz}

\makeatletter
\@namedef{subjclassname@2010}{%
\textup{2010} Mathematics Subject
Classification} \makeatother

\DeclareMathOperator*{\bigcross}{\bigCross}
\DeclareMathOperator*{\smallcross}{\smallCross}
\DeclareMathOperator*{\card}{card}

\DeclareMathOperator{\supp}{supp}

\DeclareMathOperator{\lin}{lin}
\DeclareMathOperator{\sgn}{sgn}

\numberwithin{equation}{section}

\newtheorem{theorem}{Theorem}[section]
\newtheorem{corollary}[theorem]{Corollary}
\newtheorem{prop}[theorem]{Proposition}
\newtheorem{lemma}[theorem]{Lemma}

\theoremstyle{remark}
\newtheorem{df}[theorem]{Definition}
\newtheorem{rem}[theorem]{Remark}

\newcommand*{\ascr}{\mathscr{A}}
\newcommand*{\bscr}{\mathscr{B}}

\newcommand*{\uscr}{\mathscr{U}}
\newcommand*{\hh}{\mathcal{H}}
\newcommand*{\is}[2]{\langle#1,#2\rangle}
\newcommand*{\kk}{\mathcal{K}}

\newcommand*{\natu}{\mathbb{N}}
\newcommand*{\mcal}{\mathcal{M}}

\newcommand{\bigCross}{\mathbin{\tikz [x=1.55ex,y=1.55ex,line width=.18ex,
   baseline={([yshift=-0.55ex]current bounding
   box.center)}]
   \draw (-1,-1) -- (1,1) (-1,1) -- (1,-1);}}%

\newcommand{\smallCross}{\mathbin{\tikz [x=1.1ex,y=1.1ex,line width=.15ex,
   baseline={([yshift=-0.6ex]current bounding
   box.center)}]
   \draw (-1,-1) -- (1,1) (-1,1) -- (1,-1);}}%

\newcommand*{\nul}{\mathscr{N}}
\newcommand*{\comp}{\mathbb{C}}
\newcommand*{\borel}{\mathfrak{B}}
\newcommand*{\bou}{\boldsymbol{B}}
\newcommand*{\cbb}{\comp}

\newcommand*{\D}{\mathrm{d\hspace{.1ex}}}

\theoremstyle{definition}
\newtheorem{ex}[theorem]{Example}

\newcommand*{\Le}{\leqslant}
\newcommand*{\Ge}{\geqslant}
\newcommand*{\ran}{\mathscr{R}}
\newcommand*{\ogr}{\boldsymbol{B}}
\newcommand*{\rbb}{\mathbb{R}}
\newcommand*{\zbb}{\mathbb{Z}}
\newcommand*{\tbb}{\mathbb{T}}
\def\sslim{\qopname\relax m{{\sf{s}\textrm-}lim}}
\def\wwlim{\qopname\relax m{{\sf{w}\textrm-}lim}}
   
\begin{document}

\title[Hyperrigidity I:
singly generated commutative $C^*$-algebras]
{Hyperrigidity I: singly generated commutative
$C^*$-algebras}

   \author[P. Pietrzycki and J. Stochel]{Pawe{\l} Pietrzycki and  Jan Stochel}

   \subjclass[2020]{Primary 46L07, 47A20, 47A63,
47B15; Secondary 44A60, 46G10}

   \keywords{Hyperrigidity, completely positive map,
dilations, (semi-) spectral measure, (sub-) normal
operator, weak and strong operator topologies, Choquet
boundary}

   \address{Wydzia{\l} Matematyki i Informatyki, Uniwersytet
Jagiello\'{n}ski, ul. {\L}ojasiewicza 6, PL-30348
Krak\'{o}w, Poland}

   \email{pawel.pietrzycki@im.uj.edu.pl}

   \address{Wydzia{\l} Matematyki i Informatyki, Uniwersytet
Jagiello\'{n}ski, ul. {\L}ojasiewicza 6, PL-30348
Krak\'{o}w, Poland}

   \email{jan.stochel@im.uj.edu.pl}

   \thanks{The research of both
authors was supported by the National Science Center
(NCN) Grant OPUS No.\ DEC-2021/43/B/ST1/01651.}

   \begin{abstract}
Although Arveson's hyperrigidity conjecture was
recently resolved negatively by B. Bilich and A.
Dor-On, the problem remains open for commutative
$C^*$-algebras. Relatively few examples of hyperrigid
sets are known in the commutative case. The main goal
of this paper is to determine which sets of monomials
in $t$ and $t^*$, where $t$ is a generator of a
commutative unital $C^*$-algebra, are hyperrigid. We
show that this class of hyperrigid sets has
significant connections to other areas of functional
analysis and mathematical physics. Moreover, we
develop a topological approach based on weak and
strong limits of normal (or subnormal) operators to
characterize hyperrigidity tracing back to ideas of C.
Kleski and L. G. Brown. Employing Choquet boundary
techniques, we present examples that discuss the
optimality of our results.
   \end{abstract}

   \maketitle

   \section{Introduction}
Denote by $\natu$ the set of all positive integers.
Henceforth, $\ogr(\hh)$ stands for the $C^*$-algebra
of all bounded linear operators on a (complex) Hilbert
space $\hh$, and $I$ denotes the identity operator on
$\hh$. We assume that completely positive maps are
linear and representations of unital $C^*$-algebras
preserve units and involutions.

The classical approximation theorem due to P. P.
Korovkin \cite{Kor53} states that for any sequence of
positive linear maps $\varPhi_k\colon C[0,1]\to
C[0,1]$ ($k\in \natu$),
   \begin{align*}
\lim_{k\to\infty}\|\varPhi_k(x^j)-x^j\|=0 \;\; \forall
j\in \{0,1,2\} \implies
\lim_{k\to\infty}\|\varPhi_k(f)-f\|=0 \;\; \forall
f\in C[0,1].
   \end{align*}
In other words, the asymptotic behaviour of the
sequence $\{\varPhi_k\}_{k=1}^\infty$ on the
$C^*$-algebra $C[0, 1]$ is uniquely determined by the
set $G=\{1, x, x^2\}$. Sets with this property are
called Korovkin sets. Korovkin's theorem unified many
existing approximation processes, such as the
Bernstein polynomial approximation of continuous
functions and the Fej\'{e}r trigonometric polynomial
approximation of continuous functions on the unit
circle. Another major achievement was the discovery of
geometric theory of Korovkin sets by Y. A.
\v{S}a\v{s}kin \cite{Sas67}. Namely, \v{S}a\v{s}kin
observed that the key property of $G$ is that the
Choquet boundary of the vector space spanned by $G$
coincides with $[0, 1]$ (see Section~\ref{Sec.2} for
the definition). Motivated both by the fundamental
role of the Choquet boundary in classical
approximation theory and by the importance of
approximation in the contemporary theory of operator
algebras, Arveson introduced hyperrigidity as a form
of noncommutative approximation that captures many
important operator-algebraic phenomena.
   \begin{df}[Hyperrigidity] \label{dyfnh}
A nonempty subset $ G$ of a unital $C^*$-algebra
$\ascr$ is said to be \textit{hyperrigid} ({\em
relative to} $\ascr$\/) if for any faithful
representation $\pi\colon \ascr\to \ogr(\hh)$ on a
Hilbert space $\hh$ and for any sequence
$\varPhi_k\colon \ogr(\hh)\to \ogr(\hh)$ ($k\in
\natu$) of unital completely positive (UCP) maps,
   \begin{align*}
\lim_{k\to\infty}\|\varPhi_k(\pi(g))-\pi(g)\|=0
\; \forall g\in G \implies
\lim_{k\to\infty}\|\varPhi_k(\pi(a))-\pi(a)\|=0
\; \forall a\in \ascr.
   \end{align*}
   \end{df}
Unlike Arveson's original definition of hyperrigidity
(cf.\ {\cite[Definition~1.1]{Arv11}}), we do not
require the set $G$ to be at most countable or to
generate $\ascr$. For further details, we refer the
reader to Appendix~\ref{App.B}. Let us also note that
even if $\ascr$ is commutative, the concept of
hyperrigidity is {\em a priori} stronger than that of
a Korovkin set, since in the former case, each map
$\varPhi_k$ can take noncommutative values.

Arveson provided several examples of hyperrigid sets.
In particular, he proved that if $T \in \ogr(\hh)$ is
a selfadjoint operator and $\ascr$ is the unital
$C^*$-algebra generated by $T$, then the set $G = \{I,
T, T^2 \}$ is hyperrigid in $\ascr$. Arveson asked in
\cite[Remark~9.5]{Arv11} whether the set $\{I,
T,f(T)\}$ is hyperrigid in the unital $C^*$-algebra
generated by a selfadjoint operator $T\in \ogr(\hh)$,
where $f$ is a strictly convex continuous function on
a closed subinterval of the real line $\rbb$
containing the spectrum of $T$. A positive answer was
provided by L. G. Brown in \cite{Brown16}, by using
the following theorem, which inspired us to develop a
new characterization of hyperrigidity for singly
generated commutative unital $C^*$ algebras (see
Theorem~\ref{jydgunr}). Below, the spectrum of an
operator $T\in \ogr(\hh)$ is denoted by $\sigma(T)$.
   \begin{theorem}[{\cite[Theorem~2.1]{Brown16}}]
Let $f$ be a strictly convex continuous function on an
interval $J\subseteq \rbb$, let $\hh$ be a Hilbert
space, and let $\{T_i\}_{i\in \varOmega}$ be a net of
selfadjoint operators in $\ogr(\hh)$ such that
$\sigma(T_i)\subseteq J$ for all $i\in \varOmega$. If
$\{T_i\}_{i\in \varOmega}$ converges weakly to a
selfadjoint operator $T\in \ogr(\hh)$ with
$\sigma(T)\subseteq J$, and if also $\{f(T_i)\}_{i\in
\varOmega}$ converges weakly to $f(T)$, then
$\{\varphi(T_i)\}_{i\in \varOmega}$ converges strongly
to $\varphi(T)$ for every bounded continuous function
$\varphi$ on $J$. In particular if the net
$\{T_i\}_{i\in \varOmega}$ is bounded, then
$\{T_i\}_{i\in \varOmega}$ converges strongly to $T$.
   \end{theorem}

In accordance with \v{S}a\v{s}kin's insightful
observation, Arveson \cite{Arv11} conjectured that
hyperrigidity is equivalent to the noncommutative
Choquet boundary of $ G$ being as large as possible.
This is now known as \textit{Arveson's hyperrigidity
conjecture} (see \cite[Conjecture~4.3]{Arv11}). Some
positive solutions of Arveson's hyperrigidity
conjecture have been found for certain classes of
$C^*$-algebras (see
\cite{Arv11,Kle14,CH18,Sal19,Sh20,Sch24}). However,
the conjecture ultimately has a negative solution.
Recently, a counterexample was found by B. Bilich and
A. Dor-On in \cite{Bi-Dor24} (see also
\cite{Bi24,Sch25}). However, Arveson's hyperrigidity
conjecture remains open for commutative
$C^*$-algebras, even in the singly generated~case. In
recent years, this issue has attracted considerable
interest in various parts of operator algebras and
operator theory
\cite{KS15,Clo18,Clo18b,CH18,DK21,Thom24,CH-Th24}.

This paper shows that hyperrigidity is closely tied to
characterizing when semispectral measures are
spectral, via their operator moments. The classical
von Neumann approach with Borel spectral measures on
$\rbb$ cannot capture features such as measurement
inaccuracy and measurement statistics. Hence modern
quantum theory relies on semispectral measures (see
\cite{P-S22-b} and references therein). Kiukas, Lahti,
and Ylinen \cite{KLY06JMAA,KLY06JMP} developed
quantization methods producing observables beyond
spectral measures, using operator moments, and asked:
{\em When is a positive operator measure projection
valued}\/? An early result (see
\cite{K-dM87,KLY06JMP,P-S21}) states that if $F$ is a
compactly supported Borel semispectral measure on
$\mathbb{R}$, then
   \begin{align} \label{fspyr}
\Big(\int_{\rbb} x F(\D x)\Big)^{2} = \int_{\rbb}
x^{2} F(\D x) \;\; \iff \;\; \textit{$F$ is spectral.}
   \end{align}
Some recently discovered characterizations of spectral
measures can be found in \cite[Theorem~4.2]{P-S21} and
\cite[Theorem~1.7]{P-S22-b}. These results enable
applications in operator theory, as shown by the
affirmative solution to \cite[Problem~1.1]{Curto20} in
\cite{P-S21}.
   \section{\label{Sec.m}Main results}
Relatively few examples of hyperrigid sets are known
in the case of commutative unital $C^*$-algebras. Even
proving the hyperrigidity of the set $G = \{1, x,
f(x)\}$, where $f$ is a strictly convex continuous
function on a closed subinterval of $\mathbb{R}$, or
of the operator system $A(K)$, consisting of all
continuous affine functions on a compact convex set $K
\subseteq \mathbb{R}^2$, requires considerable effort
(see \cite{Brown16, Sch24}).

In this section, we formulate the main results of the
paper. Our goal is to examine, for a commutative
unital $C^*$-algebra $\ascr$ generated by a single
element $t$, which subsets $G$ of $\big\{t^{*m}
t^n\colon (m,n) \in \zbb_+^2\big\}$ are hyperrigid in
$\ascr$, where $\zbb_+$ denotes the set of all
nonnegative integers. This choice has the advantage
that we can, to a large extent, determine when $G$
generates $\ascr$ (see Lemmas~\ref{winjkw} and
\ref{cxduns}). Another advantage is that it enables us
to provide a relatively simple criterion for
hyperrigidity. Our most general result shows that,
under an additional mild assumption, if $G$ generates
$\ascr$, then $G$ is automatically hyperrigid in
$\ascr$.
   \begin{theorem} \label{hypth}
Let $\ascr$ be a commutative unital $C^*$-algebra
generated by an element $t\in \ascr$, and let $\varXi
\subseteq \zbb_+^2$ be a set for which there exist
$p,q,r\in \zbb_+$ such that
   \begin{align*}
\text{$p\neq q$, $p+q<2r$ and $(p,q), (r,r) \in
\varXi$.}
   \end{align*}
If $G:=\{t^{*m} t^n \colon (m,n) \in \varXi\}$
generates $\ascr$, then $G$ is hyperrigid in $\ascr$.
   \end{theorem}
Theorem~\ref{hypth}, in particular, admits two
possibilities that appear in Theorem~\ref{dsaqw}
below: the first depends only on the exponents $(m,n)$
in the monomials $t^{*m} t^n$, while the second is
strongly influenced by the geometry of the spectrum
$\sigma(t)$ of~$t$. Hereafter, $\borel(X)$ denotes the
Borel $\sigma$-algebra of a topological Hausdorff
space $X$.
   \begin{theorem} \label{dsaqw}
Let $\ascr$ be a commutative unital $C^*$-algebra
generated by $t\in \ascr$, and let $\varXi \subseteq
\zbb_+^2$. Set $\varSigma=\{(p,q,r)\in \zbb_+^3\colon
p\neq q \text{ and } p+q<2r\}$. Suppose that one of
the following conditions holds:
   \begin{enumerate}
   \item[(i)] $(p,q), (r,r) \in \varXi$
for some $(p,q,r) \in \varSigma$
and\/\footnote{\label{fitupi}The abbreviation
``$\gcd$'' stands for ``the greatest common divisor``
(always assumed to be positive). It follows from the
well-ordering principle that if $J$ is a nonempty set
of integers, not all $0$, then $\gcd(J)$ exists and
there exists a finite nonempty subset $J_0$ of $J$
such that $\gcd(J) = \gcd (J_0)$. For simplicity,
writing $\gcd(J)$ means that $J$ contains a nonzero
integer.} \mbox{$\gcd \{m-n\colon (m,n)\in
\varXi\}=1$},
   \item[(ii)]
$\varXi=\{(p,q), (r,r)\}$ for some $(p,q,r) \in
\varSigma$, and there exists $\epsilon
> 0$ and disjoint sets $\{L_j\}_{j=0}^{n-1} \subseteq
\borel(\{e^{i \phi}\colon \phi \in
[0,\frac{2\pi}{n})\})$ with $n = |p-q|$ such that
   $$ \sigma(t)  \subseteq \bigcup_{j=0}^{n-1}
\Big\{\eta \omega e^{i\frac{2\pi j}{n}} \colon \eta
\in[\epsilon,\infty),\: \omega \in L_j\Big\}.
   $$
   \end{enumerate}
   Then $G:=\{t^{*m} t^n \colon (m,n) \in \varXi\}$
generates $\ascr$ and is hyperrigid in $\ascr$.
   \end{theorem}
In Theorems~\ref{hypth} and~\ref{dsaqw} there appear
at least two pairs of exponents, namely a non-diagonal
pair $(p,q)$ and a diagonal pair $(r,r)$, such that
$p+q<2r$. This raises two questions. The first is why
the condition $p+q<2r$ is required; this is discussed
in Remark~\ref{ytsq}. The second is why the presence
of a diagonal pair $(r,r)$ is necessary. As shown in
Section~\ref{Sec.9} (see Proposition~\ref{isnytos}),
if the diagonal pair $(r,r)$ is missing from $\varXi$
and $\varXi$ contains at least two distinct
non-diagonal pairs for which the corresponding set $G$
generates $C(X)$, then $G$ need not be hyperrigid in
$C(X)$.

The proofs of Theorems~\ref{hypth} and \ref{dsaqw} are
spread across Sections~\ref{Sec.6v} and~\ref{Sec.6w}.
   \section{Topological characterization of hyperrigidity}
The characterization of hyperrigidity given in this
section is inspired by the theorems of Brown and
Kleski \cite{Kle14,Brown16} and focuses on weak and
strong convergence of sequences of normal or subnormal
operators. It is also closely related to three
classical results in operator theory, namely Kadison's
theorem (see \cite[Theorem~4.2]{Ka68}), Bishop's
theorem, and the Conway-Hadwin theorem (see
\cite[Theorem~II.1.17]{Con91} and \cite{CoHad83}; for
more details, see~Appendix~\ref{AppC}).

If $X$ is a compact Hausdorff space, then $C(X)$
denotes the $C^*$-algebra of all continuous complex
functions on $X$, equipped with the supremum norm. The
functional calculus for subnormal operators is
discussed in Section~\ref{Sec.2}. We write $\wwlim$
and $\sslim$ for limits in the {\em weak} and the {\em
strong operator topologies},~respectively.
   \begin{theorem} \label{jydgunr}
Let $X$ be a nonempty compact subset of $\cbb$ and $G$
be a set of generators of $C(X)$. Then the following
conditions are equivalent{\em :}
   \begin{enumerate}
   \item[(i)] $G$ is hyperrigid,
   \item[(ii)]
for every Hilbert space $\hh$, every sequence
$\{T_n\}_{n=1}^{\infty} \subseteq \bou(\hh)$ of
subnormal operators and every normal operator $T\in
\bou(\hh)$, all with spectrum in $X$,
   \begin{align} \label{cinpt}
\wwlim_{n\to \infty} f(T_n) = f(T) \;\; \forall f\in G
\implies \sslim_{n\to \infty} f(T_n) = f(T) \;\;
\forall f\in C(X),
   \end{align}
   \item[(iii)]
for every Hilbert space $\hh$ and for all normal
operators $T, T_n \in \bou(\hh)$ $($$n\in \natu$$)$
with spectrum in $X$, implication \eqref{cinpt} holds.
   \end{enumerate}
Moreover, conditions {\em (i)-(iii)} are still
equivalent regardless of whether the Hilbert spaces
considered in either of them are separable or not.
   \end{theorem}
In Section~\ref{Sec.9} (see Theorem~\ref{main2}),
using Choquet boundary techniques, we examine the
negation of condition~(iii) of Theorem~\ref{jydgunr},
as it arises when $G$ is not hyperrigid. The proof of
Theorem~\ref{jydgunr} can be found in
Section~\ref{Sec.6}.
   \section{Applications of Theorems~\ref{hypth}
and~\ref{jydgunr}}
   We begin with a characterization of spectral
measures on the complex plane in the spirit of
\eqref{fspyr} (cf.\ \cite{P-S21} and \cite{P-S22-b}).
We refer to Example~\ref{isuszn}, which illustrates
the importance of the normality of $T$.
   \begin{theorem} \label{specth}
Let $T\in \ogr(\hh)$ be a normal operator on a Hilbert
space $\hh$, let $F\colon \borel(\cbb) \to \ogr(\hh)$
be a semispectral measure with compact support, and
let $\varXi \subseteq \zbb_+^2$ be as in Theorem~{\em
\ref{dsaqw}(i)}. Then the following conditions are
equivalent{\em :}
   \begin{enumerate}
   \item[(i)] $F$ is the spectral measure of $T$,
   \item[(ii)] $T^{*m}T^{n}=\int_{\cbb} \bar{z}^{m} z^{n} F(\D z)$
for all $(m,n)\in \varXi$.
   \end{enumerate}
   \end{theorem}
Applying Theorem~\ref{specth} to
$\varXi=\big\{(\frac{p-1}{2},\frac{p+1}{2}),
(\frac{q}{2},\frac{q}{2})\big\}$
gives~\cite[Theorem~1.7]{P-S22-b}.
   \begin{theorem}\label{sotwotth}
Let $X$ be a nonempty compact subset of $\cbb$, let
$\{T_k\}_{k=1}^{\infty} \subseteq \bou(\hh)$ be a
sequence of subnormal operators with spectrum in $X$,
and let $T\in \bou(\hh)$ be a normal operator with
spectrum in $X$. Let $\varXi\subseteq \zbb_+^2$.
Suppose that one of the conditions {\em (i)-(ii)} of
Theorem~{\em \ref{dsaqw}} holds with $X$ in place of
$\sigma(t)$, and that
   \begin{align*}
\wwlim_{k\to\infty} T_k^{*m}T_k^n & = T^{*m}T^n, \quad
(m,n)\in \varXi.
   \end{align*}
Then
   \begin{align} \label{dycnim}
\sslim_{k\to \infty} f(T_k) = f(T), \quad f\in C(X).
   \end{align}
   \end{theorem}
Removing the condition that enforces a specific
location of the spectrum from Theorem~\ref{sotwotth}
leads to the following result.
   \begin{theorem}
\label{wszb} Let $\{T_k\}_{k=1}^{\infty} \subseteq
\bou(\hh)$ be a sequence of subnormal operators, and
let $T\in \bou(\hh)$ be a normal operator. Suppose
that $\varXi\subseteq \zbb_+^2$ is as in Theorem~{\em
\ref{dsaqw}(i)}~and that
   \begin{align*}
\wwlim_{k\to\infty} T_k^{*m}T_k^n = T^{*m}T^n, \quad
(m,n)\in \varXi.
   \end{align*}
Then
   \begin{align} \label{woso}
\sslim_{k\to\infty} T_k^{*m}T_k^n= T^{*m}T^n, \quad
(m,n)\in \zbb_+^2.
   \end{align}
   \end{theorem}
The proofs of Theorems~\ref{specth}, \ref{sotwotth}
and \ref{wszb} can be found in Section~\ref{Sec.7n}.
   \section{\label{Sec.2}Prerequisites}
In this paper, we use the following notation. The
fields of real and complex numbers are denoted by
$\rbb$ and $\mathbb{C}$, respectively. The symbols
$\zbb$, $\zbb_{+}$, $\natu$ and $\rbb_+$ stand for the
sets of integers, nonnegative integers, positive
integers and nonnegative real numbers, respectively.
Denote by $\delta_{m,n}$ the Kronecker delta. The
cardinality of a set $X$ is denoted by $\card(X)$. We
write $\borel(X)$ for the $\sigma$-algebra of all
Borel subsets of a topological Hausdorff space $X$. If
$\mu$ is a finite (positive) Borel measure on $\cbb$,
the closed support of $\mu$ is denoted by
$\supp(\mu)$.

Let $\ascr$ be a unital $C^*$-algebra with unit $e$.
Given a nonempty subset $G$ of $\ascr$ (not
necessarily containing $e$), we denote by $C^*(G)$ the
smallest unital $C^*$-subalgebra of $\ascr$ containing
$G$, or equivalently, the smallest (not necessarily
unital) $C^*$-subalgebra of $\ascr$ containing $G\cup
\{e\}$. If $t\in \ascr$, then we write
$C^*(t)=C^*(\{t\})$. In case $\ascr=C^*(G)$, we say
that $G$ {\em generates} $\ascr$, or that $G$ is a set
of {\em generators} of $\ascr$. The spectrum of an
element $t$ of $\ascr$ is denoted by $\sigma(t)$. We
say that $\mcal \subseteq \ascr$ is an {\em operator
system} if $\mcal$ is a vector subspace of $\ascr$
containing the unit and having the property that
$a^*\in \mcal$ for all $a\in \mcal$. If $\mcal$ is an
operator system which is separable with respect to the
norm topology, then $\mcal$ is called a {\em separable
operator system}. Note that if $ G\subseteq \ascr$ is
an at most countable set of generators of $\ascr$,
then the linear span $\mcal$ of $\{e\} \cup G \cup
G^*$ with $G^*:=\{a^*\colon a \in G\}$ is a separable
operator system generating $\ascr$. And {\em vice
versa}, if $\mcal$ is a separable operator system
generating $\ascr$, then there exists an at most
countable subset $G$ of $\mcal$ generating $\ascr$. A
unital $C^*$-algebra $\ascr$ is separable if and only
if $\ascr$ is generated by an at most countable set,
or equivalently, by a separable operator system. We
abbreviate the phrase ``unital completely positive
map'' to ``UCP map'', where ``unital'' means that the
map preserves units. If $\ascr$ is a commutative
unital $C^*$-algebra, we denote by
$\mathfrak{M}_{\ascr}$, $\borel(\mathfrak{M}_{\ascr})$
and $\widehat{a}\colon \mathfrak{M}_{\ascr} \to \cbb$
the maximal ideal space of $\ascr$, the
$\sigma$-algebra of Borel subsets of
$\mathfrak{M}_{\ascr}$ and the Gelfand transform of
$a\in \ascr$, respectively.

Let $X$ be a compact Hausdorff space. For $\lambda \in
X$, $\delta_{\lambda}$ denotes the {\em Dirac measure}
at $\lambda$, i.e. the Borel probability measure on
$X$ concentrated on $\{\lambda\}$. Suppose that
$\mathcal{M}$ is a subspace of $C(X)$ and that $1\in
\mathcal{M}$. Denote by $\mathcal{M}^{\prime}$ the
Banach space of all continuous linear functionals on
$\mathcal{M}$. The \textit{state space}
$S(\mathcal{M})$ of $\mathcal{M}$ is the set of all
$L$ in $\mathcal{M}^{\prime}$ such that $L(1) = 1 =
\|L\|$. Let $\partial_\mathcal{M}X$ be the set of all
$\lambda\in X$ for which $\delta_{\lambda}$ is an
extreme point of $S(\mathcal{M})$, where now
$\delta_{\lambda}$ is understood as the element of
$\mathcal{M}^{\prime}$ defined by
$\delta_{\lambda}(f)=f(\lambda)$ for $f\in
\mathcal{M}$ (in fact, $\delta_{\lambda}(f)=\int_X f
\D \delta_{\lambda}$). We call $\partial_\mathcal{M}X$
the \textit{Choquet boundary} for $\mathcal{M}$. We
also say that a regular Borel probability measure
$\mu$ on $X$ {\em represents} a point $\lambda\in X$
(relative to $\mcal$) if $f(\lambda)=\int_X f \D \mu$
for every $f\in \mcal$. The Dirac measure
$\delta_{\lambda}$ represents $\lambda$. In view of
\cite[Proposition~6.2]{Phe02}, if $\mcal$ separates
the points of $X$ and $1\in \mathcal{M}$, then a point
$\lambda\in X$ is in $\partial_{\mcal}X$ if and only
if $\mu=\delta_{\lambda}$ is the only measure
representing $\lambda$. For an overview of Choquet
theory, see~\cite{Phe02}.

If $Z$ is a subset of a vector space, $\lin Z$ denotes
the linear span of $Z$. Let $\hh$ and $\kk$ be Hilbert
spaces. Denote by $\ogr(\hh, \kk)$ the Banach space of
all bounded linear operators from $\hh$ to $\kk$
equipped with the operator norm. If $A\in
\ogr(\hh,\kk)$, then $A^*$, $\nul(A)$ and $\ran(A)$
stand for the adjoint, the kernel and the range of
$A$, respectively. The space
$\ogr(\hh):=\ogr(\hh,\hh)$ is a $C^*$-algebra with
unit $I=I_{\hh}$, the identity operator on $\hh$. We
write $\sigma(A)$ for the spectrum of $A \in
\ogr(\hh)$. We say that $A\in \ogr(\hh)$ is
\textit{normal} if $A^*A=AA^*$, \textit{selfadjoint}
if $A=A^*$ and \textit{positive} if $\langle
Ah,h\rangle \Ge 0$ for all $h\in \hh$. If $A\in
\ogr(\hh)$, we write $|A|:=(A^*A)^{1/2}$. Recall that
$\nul(A)=\nul(|A|)$. Every operator $A\in \ogr(\hh)$
has the {\em polar decomposition} $A=V|A|$, where $V$
is a (unique) partial isometry such that
$\nul(V)=\nul(A)$. If $N \in \ogr(\hh)$ is normal,
then there exists a unique $U\in \ogr(\hh)$ such that
(see \cite[Theorem~6, p.\ ~66]{Fur01})
   \begin{align} \label{pdu}
\text{$N=U |N|=|N|U$, \, $U$ is a unitary operator \,
and \, $U|_{\nul(N)}=I_{\nul(N)}$.}
   \end{align}
If \eqref{pdu} holds, then we say that $N=U |N|$ is
the \textit{unitary polar decomposition} of $N$.
Comparing both polar decompositions for a normal
operator $N$, we see that
$\overline{\ran(N)}=\overline{\ran(|N|)}$, $\tilde
V:=V|_{\overline{\ran(N)}}$ is unitary and
   \begin{align} \label{orotrh}
V=\tilde V \oplus 0, \quad U=\tilde V \oplus
I_{\nul(N)},
   \end{align}
where $0$ stands for the zero operator on $\nul(N)$.

We say that $S\in \ogr(\hh)$ is {\em subnormal} if
there exists a Hilbert space $\kk$ and a normal
operator $N\in \ogr(\kk)$, called a {\em normal
extension} of $S$, such that $\hh\subseteq \kk$
(isometric embedding) and $Sh=Nh$ for all $h\in \hh$.
If $\kk$ is the only closed subspace that reduces $N$
and contains $\hh$, we call $N$ a {\em minimal} normal
extension of~$S$.

Let $\mathfrak{A}$ be a $\sigma$-algebra of subsets of
a set $X$ and $F\colon \mathfrak{A} \to \ogr(\hh)$ be
a semispectral measure, that is, $\is{F(\cdot)f}{f}$
is a (positive) measure on $\mathfrak{A}$ for every $f
\in \hh$, and $F(X)=I$. Denote by $L^1(F)$ the vector
space of all $\mathfrak{A}$-measurable functions
$f\colon X \to \cbb$ such that $\int_{X} |f(x)|
\langle F(\D x)h, h\rangle < \infty$ for all $h\in
\hh$. Then for every $f\in L^1(F)$, there exists a
unique operator $\int_X f \D F \in \ogr(\hh)$ such
that (see, e.g., \cite[Appendix]{Sto92})
   \begin{align*}
\Big\langle\int_X f \D F h, h\Big\rangle =
\int_X f(x) \langle F(\D x)h, h\rangle,
\quad h\in\hh.
   \end{align*}
If $F$ is the spectral measure of a normal operator
$A\in \ogr(\hh)$, then we write $f(A)=\int_{\cbb} f \D
F$ for any $F$-essentially bounded Borel function
$f\colon \cbb \to \cbb$; the map $f \mapsto f(A)$ is
called the Stone-von Neumann calculus. For more
information needed in this article on spectral
integrals, including the spectral theorem for normal
operators and the Stone-von Neumann calculus, we refer
the reader to \cite{Rud73,Bir-Sol87}.

Recall the functional calculus for subnormal
operators. Suppose that $S\in \ogr(\hh)$ is a
subnormal operator and $N\in \ogr(\kk)$ is its minimal
normal extension. Let $E\colon \borel(\cbb) \to
\ogr(\kk)$ be the spectral measure of $N$, and $P\in
\ogr(\kk)$ be the orthogonal projection of $\kk$ onto
$\hh$. The semispectral measure $F\colon \borel(\cbb)
\to \ogr(\hh)$ defined by $F(\varDelta) =
PE(\varDelta)|_{\hh}$ for $\varDelta \in\borel(\cbb)$,
is called the {\em semispectral measure} of $S$. By
\cite[Proposition~5]{Ju-St08} and
\cite[Proposition~II.2.5]{Con91}, the definition of
$F$ does not depend on the choice of $N$. For compact
$X \subseteq \cbb$ with\footnote{\,As usual,
$\supp(F)$ denotes the closed support of $F$.}
$\supp(F) \subseteq X$ and $f \in C(X)$, we set
   \begin{align} \label{fysk}
f(S)=\int_{X} f \D F.
   \end{align}
Note that $\supp(F)=\supp(E)=\sigma(N)$ (see
\cite[Proposition~4(iii)]{Ju-St08}; see also
\cite[Theorem~4.4]{Ja02}). Since $\sigma(N)\subseteq
\sigma(S)$ (see \cite[Theorem~II.2.11]{Con91}), we
deduce that \eqref{fysk} makes sense for every $f \in
C(X)$ whenever $X\subseteq \cbb$ is compact and
$\sigma(S) \subseteq X$. In case $F$ is a spectral
measure (and consequently $S$ is normal), the map $f
\mapsto f(S)$ coincides with the Stone-von Neumann
calculus. Since $S^n=N^n|_{\hh}$ and
$S^{*n}=PN^{*n}|_{\hh}$ for all $n\in \zbb_+$, an
application of the Stone-von Neumann calculus yields
   \begin{equation} \label{add-2}
S^{*m}S^n = PN^{*m}N^n|_{\hh} = \int_{X}
\bar z^m z^{n} F(\D z) = (\bar\xi^m
\xi^n)(S), \quad m,n \in \zbb_+,
   \end{equation}
where $\xi\in C(X)$ is defined by
   \begin{align} \label{ksik}
\xi(z) = z, \quad z \in X.
   \end{align}
   \begin{lemma} \label{rudec}
Let $T\in \ogr(\hh)$ and let $P \in
\ogr(\hh)$ be an orthogonal projection. Then
the following conditions are
equivalent{\em:}
   \begin{itemize}
   \item[(i)]
$(PT^*P)(PTP) = PT^*TP$ and $(PTP)(PT^*P) =
PTT^*P$,
   \item[(ii)] $PT=TP$,
   \item[(iii)] $\ran(P)$ reduces $T$.
   \end{itemize}
Moreover, if $T$ and $PTP$ are normal, then
{\em (ii)} is equivalent to
   \begin{itemize}
\item[(iv)] $(PT^*P)(PTP) = PT^*TP$.
   \end{itemize}
   \end{lemma}
   \begin{proof}
Equivalence (ii)$\Leftrightarrow$(iii) and implication
(ii)$\Rightarrow$(i) are obvious.

(i)$\Rightarrow$(ii) By \cite[Lemma~3.2]{P-S-roots23},
$TP=PTP$ and $T^*P=PT^*P$, so $PT=TP$.
   \end{proof}
Let $J\subseteq \rbb$ be an interval (open, half-open,
or closed; bounded or unbounded). A continuous
function $f \colon J \rightarrow \rbb$ is said to be
\textit{operator monotone} if $f(A)\Le f(B)$ for any
two selfadjoint operators $A,B\in\ogr(\hh)$ such that
$A\Le B$ and the spectra of $A$ and $B$ are contained
in $J$ (see \cite{Sim19}). The most important example
of an operator monotone function is $f\colon
[0,\infty)\ni t\rightarrow t^p\in\rbb$, where $p\in
(0,1)$. Operator monotone functions are related to the
Hansen inequality \cite{Han80}. In
\cite[Lemma~2.2]{uch93} (see also \cite[p.\
6]{P-S21}), Uchiyama gave a necessary and sufficient
condition for equality to hold in the Hansen
inequality when the external factor is a nontrivial
orthogonal~projection.
   \begin{theorem}[\cite{Han80,uch93}]
\label{hans} Let $A\in \ogr(\hh)$ be a
positive operator, $T \in \ogr(\hh)$ be a
contraction and $f\colon
[0,\infty)\rightarrow \rbb$ be a continuous
operator monotone function such that
$f(0)\Ge 0$. Then
   \begin{equation} \label{Han-inq}
T^*f(A)T \Le f(T^*AT).
   \end{equation}
Moreover, if $f$ is not an affine function
and $T$ is an orthogonal projection such
that $T \neq I$, then equality holds in
\eqref{Han-inq} if and only if $TA=AT$ and
$f(0)=0$.
   \end{theorem}
A linear map $\varPhi \colon \mathcal{A} \rightarrow
\mathcal{ B}$ between unital $C^*$-algebras is said to
be {\em positive} if $\varPhi(a) \Ge 0$ for all $a\in
\mathcal{A}$ such that $a \Ge 0$. We also need the
Schwarz-type inequality.
   \begin{theorem}[{\cite[Theorem 2]{L-R74}}] \label{L-R}
Let $R\in \ogr(\hh,\kk)$ and let
$\varPhi\colon\ogr(\kk)\to\ogr(\hh)$ be the positive
linear map defined by
   \begin{align} \label{obmiar}
\varPhi(X)=R^*XR,\quad X\in \ogr(\kk).
   \end{align}
Then for all $A,B\in\ogr(\kk)$, the net
$\{\varPhi(A^*B)(\varPhi(B^*B)+\varepsilon
I)^{-1}\varPhi(B^*A)\}_{\varepsilon>0}$ is
convergent in the strong operator topology
as $\varepsilon \downarrow 0$ and
   \begin{align*}
\varPhi(A^*A)\Ge
\sslim_{\varepsilon\downarrow 0}\,
\varPhi(A^*B)(\varPhi(B^*B)+\varepsilon
I)^{-1}\varPhi(B^*A).
   \end{align*}
   \end{theorem}
   \section{\label{Sec.6v}Preparations for the proof of Theorem~\ref{hypth}}
   One of the main tools used in the proof of
Theorem~\ref{hypth} is stated below.
   \begin{theorem}\label{main1}
Let $\hh$ be a closed subspace of a Hilbert space
$\kk$, $T\in \bou(\hh)$ and $N\in \ogr(\kk)$ be normal
operators and $p,q,r\in \zbb_+$ be such that $p<q$ and
$\frac{p+q}{2}<r$. Suppose that $\varXi$ is a subset
of $\zbb_+^2$ such that $\{(p,q), (r,r)\} \subseteq
\varXi$ and
   \begin{equation}  \label{pukur}
T^{*m}T^n=PN^{*m}N^n|_\hh, \quad (m,n)\in \varXi,
   \end{equation}
where $P\in \ogr(\kk)$ is the orthogonal projection of
$\kk$ onto $\hh$. Then
   \begin{itemize}
    \item[(i)] $T^{*n}T^n=PN^{*n}N^n|_\hh$ for all $n\in
\zbb_+$,
    \item[(ii)] $P$ commutes with
$N^{d}$, where $d:=\gcd \{m-n\colon (m,n)\in
\varXi\};$ equivalently{\em :} $P$ commutes with $|N|$
and $U^{d}$, where $N=U|N|$ is the unitary polar
decomposition of $N$ $($see \eqref{pdu}$)$.
   \end{itemize}
   \end{theorem}
Before proving Theorem~\ref{main1}, we present a few
lemmas.
   \begin{lemma} \label{npq}
Let $N\in \ogr(\hh)$ be a normal operator,
$N=U |N|$ be the unitary polar decomposition
of $N$ $($see \eqref{pdu}$)$, $P \in
\ogr(\hh)$ be an orthogonal projection and
$(p,q)\in \zbb_+^2 \setminus \{(0,0)\}$.
Then the following conditions are
equivalent{\em :}
   \begin{enumerate}
   \item[(i)] $P$ commutes with $N^{*p}N^q$,
   \item[(ii)] $P$ commutes with $|N|$ and
$U^{q-p}$.
   \end{enumerate}
Moreover, if $p\neq q$, then {\em (i)} is
equivalent to any of the following
statements{\em :}
   \begin{enumerate}
   \item[(iii)] $P$ commutes with
$N^{*{\tilde p}}N^{\tilde q}$ for every
$(\tilde p, \tilde q) \in \zbb_+^2 \setminus
\{(0,0)\}$ such that $|\tilde q-\tilde
p|=|q-p|$,
   \item[(iv)] $P$ commutes with $N^{|q-p|}$.
   \end{enumerate}
   \end{lemma}
   \begin{proof}
Since $U$ and $|N|$ commute (see
\eqref{pdu}), we see that
   \begin{align} \label{pulder}
N^{*p}N^q=U^{q-p}|N|^{p+q}.
   \end{align}

(i)$\Rightarrow$(ii) First, we show that $P$ and $|N|$
commute. It follows from the selfadjointness of $P$
that $P$ commutes with $(N^{*p}N^q)^*$, and therefore
also with
   \begin{equation*}
(N^{*p}N^q)^*N^{*p}N^q=|N|^{2(p+q)}.
   \end{equation*}
Since $p+q \Ge 1$, we see that the commutants of
$|N|^{2(p+q)}$ and $|N|$ coincide, which yields
$P|N|=|N|P$. This, together with (i) and
\eqref{pulder}, leads to
   \begin{equation} \label{upn}
PU^{q-p}|N|^{p+q}=U^{q-p}|N|^{p+q}P=U^{q-p}P|N|^{p+q}.
   \end{equation}
The kernel-range decomposition, together with $p+q \Ge
1$, yields
   \begin{equation*}
\overline{\ran(|N|^{p+q})}^\perp =
\nul(|N|^{p+q})=\nul(|N|) =
\overline{\ran(|N|)}^\perp,
   \end{equation*}
which implies that
$\overline{\ran(|N|^{p+q})} =
\overline{\ran(|N|)}$. This, combined with
\eqref{upn}, shows that
   \begin{equation}\label{pur}
PU^{q-p}|_{\overline{\ran(|N|)}}=U^{q-p}P|_{\overline{\ran(|N|)}}.
   \end{equation}
On the other hand, the fact that $P$
commutes with $|N|$ implies that
$P\nul(|N|)\subseteq \nul(|N|)$. Since by
\eqref{orotrh},
$U|_{{\nul(|N|)}}=I_{{\nul(|N|)}}$, we get
   \begin{equation*}
PU^{q-p}|_{{\nul(|N|)}}=P|_{{\nul(|N|)}}=U^{q-p}P|_{{\nul(|N|)}}.
   \end{equation*}
This combined with \eqref{pur} shows that
$P$ commutes with $U^{q-p}$.

(ii)$\Rightarrow$(i) Use \eqref{pulder}.

Assume now that $p\neq q$.

(ii)$\Rightarrow$(iii) Apply \eqref{pulder} to
$(\tilde p, \tilde q)$ in place of $(p,q)$.

(iii)$\Rightarrow$(iv) Apply (iii) to
$\tilde p= 0$ and $\tilde q=|q-p|$.

(iv)$\Rightarrow$(ii) Apply implication
(i)$\Rightarrow$(ii) to the pair $(0,|q-p|)$
in place of $(p,q)$.
   \end{proof}
   \begin{lemma} \label{nwer}
Let $\hh$ be a closed subspace of a Hilbert
space $\kk$, $T\in \bou(\hh)$ and $N\in
\ogr(\kk)$ be normal operators and $p\in
\natu$ be such that
   \begin{equation} \label{tpq-a}
T^{*p}T^p=PN^{*p}N^p|_\hh,
   \end{equation}
where $P \in \ogr(\kk)$ is the orthogonal
projection of $\kk$ onto $\hh$. Let
$P|N|=|N|P$. Then
   \begin{equation*}
T^{*n}T^n=PN^{*n}N^n|_\hh,\quad n\in \zbb_+.
   \end{equation*}
   \end{lemma}
   \begin{proof}
Since $\hh$ reduces $|N|$, and both
operators $T$ and $N$ are normal, we get
   \begin{align*}
|T|^{2p} = T^{*p}T^p \overset{\eqref{tpq-a}}
= P N^{*p}N^p |_{\hh} = P|N|^{2p}|_{\hh} =
(|N|\big|_{\hh})^{2p}.
   \end{align*}
By uniqueness of positive $n$th roots, we see that
$|T|=|N|\big|_{\hh}$, which yields
   \begin{align*}  \tag*{\qedhere}
T^{*n}T^n = |T|^{2n} = (|N|\big|_{\hh})^{2n}
= PN^{*n}N^n|_{\hh}, \quad n \in \zbb_+.
   \end{align*}
   \end{proof}
   \begin{lemma} \label{wurnik}
Let $\hh$ be a closed subspace of a Hilbert
space $\kk$, $T\in \bou(\hh)$ and $N\in
\ogr(\kk)$ be normal operators and $p,q\in
\zbb_+$ be such that $p< q$ and
   \begin{equation} \label{tpq}
T^{*m}T^n=PN^{*m}N^n|_\hh, \quad (m,n) \in
\big\{(p,q),(q,q)\big\},
   \end{equation}
where $P \in \ogr(\kk)$ is the orthogonal
projection of $\kk$ onto $\hh$. Then
   \begin{itemize}
   \item[(i)] $T^{*n}T^n=PN^{*n}N^n|_\hh$ for
all $n\in \zbb_+$,
   \item[(ii)] $P$ commutes with $|N|$, $U^{q-p}$
and $N^{q-p}$, where $N=U |N|$ is the
unitary polar decomposition of $N$ $($see
\eqref{pdu}$)$,
  \item[(iii)] $\hh=(\hh\cap\nul(N))
\oplus(\hh\cap\overline{\ran(N)})$.
   \end{itemize}
   \end{lemma}
   \begin{proof}
Using the uniqueness of positive $n$th roots
of positive operators, we can assume without
loss of generality that $P \neq I_{\kk}$.

(i)\&(ii) We begin by proving that $P$
commutes with $|N|$. We will first consider
the case where $p\in \natu$. Let
$\varPhi\colon \ogr(\kk)\to \ogr(\hh)$ be as
in \eqref{obmiar} with $R\in \ogr(\hh,\kk)$
defined by $Rh=h$ for $h\in \hh$. Clearly
$R^*f=Pf$ for $f \in \kk$. Applying
Theorem~\ref{L-R} to $A=N^p$ and $B=N^q$
yields
   \begin{align}\label{5gw}
\varPhi(N^{*p}N^p) \Ge
\sslim_{\varepsilon\downarrow 0} \,
\varPhi(N^{*p}N^q)(\varPhi(N^{*q}N^q)+\varepsilon
I)^{-1} \varPhi(N^{*q}N^p).
   \end{align}
Let $E_{T}$ be the spectral measure of $T$.
The Stone-von Niemann calculus yields
   \allowdisplaybreaks
   \begin{align} \notag
\varPhi(N^{*p}N^q)(\varPhi(N^{*q}N^q)+\varepsilon
I)^{-1}\varPhi(N^{*q}N^p) &
\overset{\eqref{tpq}} =
T^{*p}T^q(T^{*q}T^{q}+\varepsilon
I)^{-1}T^{*q}T^p
   \\ \label{kra-kru}
& \hspace{1ex} =\int_\cbb
\frac{|z|^{2(p+q)}}{|z|^{2q}+\varepsilon}E_{T}(\D
z), \quad \varepsilon > 0.
   \end{align}
Applying Lebesgue's monotone convergence
theorem, we deduce that (use $p\Ge 1$)
   \allowdisplaybreaks
   \begin{align} \notag
\lim_{\varepsilon\downarrow 0} &
\langle\varPhi(N^{*p}N^q)(\varPhi(N^{*q}N^q)+\varepsilon
I)^{-1}\varPhi(N^{*q}N^p)h, h\rangle
   \\ \notag
& \hspace{.2ex}\overset{\eqref{kra-kru}} =
\lim_{\varepsilon\downarrow 0} \int_\cbb
\frac{|z|^{2(p+q)}}{|z|^{2q}+\varepsilon}\langle
E_{T}(\D z)h,h\rangle
   \\  \label{lyncylk}
& \hspace{1.2ex} =\int_\cbb {|z|^{2p}}\langle E_{T}(\D
z)h, h\rangle = \langle |T|^{2p}h,h\rangle, \quad h
\in \hh.
   \end{align}
Combined with \eqref{5gw}, this implies that
   \begin{align} \label{trepi}
|T|^{2p}\Le \varPhi(N^{*p}N^p).
   \end{align}
Let $\hat T\in \ogr(\kk)$ be the normal
operator defined by $\hat T = T \oplus 0$,
where $0$ is the zero operator on
$\kk\ominus \hh$. Then by Theorem~\ref{hans}
applied to the operator monotone function
$f(t)=t^{\frac{p}{q}}$, we see that
   \begin{align*}
|\hat T|^{2p} \overset{\eqref{trepi}}{\Le}
PN^{*p}N^p P = P (|N|^{2q})^{\frac{p}{q}}P
\Le (P|N|^{2q}P)^\frac{p}{q}
\overset{\eqref{tpq}}= |\hat T|^{2p},
   \end{align*}
and hence $P (|N|^{2q})^{\frac{p}{q}}P =
(P|N|^{2q}P)^\frac{p}{q}$. Because $0<\frac{p}{q} <
1$, the ``moreover'' part of Theorem~\ref{hans}
implies that $P$ commutes with $|N|^{2q}$. Since
$|N|^{2q}$ and $|N|$ have the same commutant, $P$
commutes with $|N|$ whenever $p\in \natu$.

Next, we show that $P$ commutes with $U^{q-p}$ still
under the assumption that $p\in \natu$. It follows
from Lemma~\ref{nwer} and \eqref{tpq} that statement
(i) is valid in the case where $p\in \natu$. Combined
with \cite[Lemma~3.2]{P-S-roots23}, this yields
   \begin{align} \notag
\hat T^{*(p+q)} \hat T^{p+q} & = (\hat T^{*p} \hat
T^{q})(\hat T^{*q} \hat T^{p}) \overset{\eqref{tpq}} =
(P N^{*p} N^{q}P)(P N^{*q} N^{p}P)
   \\ \label{snug}
&\Le P N^{*p} N^{q}N^{*q} N^{p}P = P N^{*(p+q)}
N^{p+q}P \overset{\mathrm{(i)}} = \hat T^{*(p+q)} \hat
T^{p+q}.
   \end{align}
Since by \eqref{tpq}, $P N^{*q} N^{p}P=\hat T^{*q}
\hat T^{p}$, we see that the operator $P N^{*q}
N^{p}P$ is normal. It follows from \eqref{snug} and
Lemma~\ref{rudec} that $P$ commutes with $N^{*p}
N^{q}$. By Lemma~\ref{npq}, $P$ commutes with
$U^{q-p}$ and $N^{q-p}$, so (ii) holds in the case
where~$p\in \natu$.

If $p=0$, then arguing as in \eqref{snug},
we get
   \begin{align*}
\hat T^{*q} \hat T^q \overset{\eqref{tpq}} = (P
N^{*q}P)(P N^{q}P) \Le P N^{*q}N^{q}P
\hspace{.8ex}\overset{\eqref{tpq}} = \hat T^{*q} \hat
T^{q}.
   \end{align*}
Hence, by Lemma~\ref{rudec}, $P$ commutes
with $N^q$, so by Lemma~\ref{npq}, $P$
commutes with $|N|$, $U^{q}$ and $N^{q}$,
and thus (ii) is also valid in the case
where $p=0$. This, together with \eqref{tpq}
and Lemma~\ref{nwer} applied to $q$ in place
of $p$, implies that (i) holds as well in
the case where $p=0$.

In summary, we have shown that statements
(i) and (ii) are valid without additional
restrictions on $p$.

(iii) By (ii), $P\nul(|N|)\subseteq
\nul(|N|)$. However, $\nul(N)=\nul(|N|)$, so
$P\nul(N)\subseteq \nul(N)$. Therefore
$\nul(N)$ reduces $P$, or equivalently
$P_{\nul(N)}P = PP_{\nul(N)}$, where
$P_{\nul(N)}\in \ogr(\kk)$ is the orthogonal
projection of $\kk$ onto $\nul(N)$. This
implies that $P_{\nul(N)}P$ and
$(I-P_{\nul(N)})P$ are orthogonal
projections of $\kk$ onto $\hh\cap\nul(N)$
and $\hh\cap\overline{\ran(N)}$,
respectively, which yields (iii).
   \end{proof}
   \begin{proof}[Proof of Theorem~\ref{main1}]
First, we show that $P$ commutes with $|N|$.
We consider three cases that cover all
logical possibilities. We can also assume,
without loss of generality, that $P \neq
I_{\kk}$.

{\sc Case 1.} $p+q>r$.

First, we show that
   \begin{align} \label{ogra}
|T|^{2(p+q-r)} \Le P |N|^{2(p+q-r)}|_\hh.
   \end{align}
For, let $\varPhi\colon \ogr(\kk)\to
\ogr(\hh)$ be as in \eqref{obmiar} with
$R\in \ogr(\hh,\kk)$ defined by $Rh=h$ for
$h\in \hh$. Applying Theorem~\ref{L-R} to
$A=U^{p- q}|N|^{p+q-r}$ and $B=|N|^r$, we
get
   \begin{align*}
\varPhi(|N|^{2(p+q-r)})\Ge
\sslim_{\varepsilon\downarrow 0}\,
\varPhi(N^{*p}N^q)(\varPhi(|N|^{2r})+\varepsilon
I)^{-1}\varPhi(N^{*q}N^p).
   \end{align*}
Therefore, using \eqref{pukur} and the
Stone-von Neumann calculus, we obtain
   \begin{align*}
P |N|^{2(p+q-r)}|_\hh \Ge
\sslim_{\varepsilon\downarrow 0}\,
T^{*p}T^q(|T|^{2r}+\varepsilon I)^{-1}T^{*q}T^p =
\sslim_{\varepsilon \downarrow 0} \int_\comp
\frac{|z|^{2(p+q)}} {|z|^{2r}+\varepsilon} E_{T}(\D
z),
   \end{align*}
where $E_{T}$ is the spectral measure of $T$. Arguing
as in \eqref{lyncylk}, we get \eqref{ogra}.

Let $\hat T\in \ogr(\kk)$ be the normal operator
defined by $\hat T = T \oplus 0$, where $0$ is the
zero operator on $\kk\ominus \hh$. By our assumptions,
$0 < \frac{p+q-r}{r}<1$. Hence, we can apply
Theorem~\ref{hans} to the operator monotone function
$f(t)=t^{\frac{p+q-r}{r}}$ getting
   \allowdisplaybreaks
   \begin{align*}
|\hat{T}|^{2(p+q-r)} & \overset{\eqref{ogra}}{\Le} P
|N|^{2(p+q-r)}P =P (|N|^{2r})^\frac{p+q-r}{r}P
   \\
& \hspace{.5ex} \overset{\eqref{Han-inq}}{\Le} (P
|N|^{2r}P)^\frac{p+q-r}{r} \overset{\eqref{pukur}} =
(|\hat{T}|^{2r})^\frac{p+q-r}{r} =
|\hat{T}|^{2(p+q-r)}.
   \end{align*}
Since equality holds in the Hansen
inequality \eqref{Han-inq}, the ``moreover''
part of Theorem~\ref{hans} implies that $P$
commutes with $|N|^{2r}$ and consequently
with $|N|$.

 {\sc Case 2.} $p+q < r$.

It follows from \cite[Lemma~3.2]{P-S-roots23} and
Theorem \ref{hans} applied to the operator monotone
function $f(t)=t^{\frac{p+q}{r}}$ that (cf.\
\eqref{snug})
   \allowdisplaybreaks
   \begin{align} \notag
|\hat{T}|^{2(p+q)} & =(\hat T^{*p}\hat
T^q)(\hat T^{*q} \hat T^p)
\overset{\eqref{pukur}} =
(PN^{*p}N^qP)(PN^{*q}N^p P)
    \\ \notag
& \Le PN^{*p}N^qN^{*q}N^pP=
P(|N|^{2r})^{\frac{p+q}{r}}P
   \\  \label{jwes}
& \hspace{-1.1ex} \overset{\eqref{Han-inq}}\Le
(P|N|^{2r}P)^{\frac{p+q}{r}} \overset{\eqref{pukur}} =
(|\hat{T}|^{2r})^{\frac{p+q}{r}} = |\hat{T}|^{2(p+q)}.
   \end{align}
This means that equality holds in the Hansen
inequality \eqref{Han-inq}. Hence applying
the ``moreover'' part of Theorem~\ref{hans},
we see that $P$ commutes with $|N|^{2r}$ and
consequently with $|N|$.

{\sc Case 3.} $p+q=r$.

Arguing as in \eqref{jwes}, we deduce that
   \begin{align} \notag
|\hat{T}|^{2(p+q)} & \overset{\eqref{pukur}}
= (PN^{*p}N^qP) (PN^{*q}N^pP)
   \\  \label{fdsa}
& \hspace{1ex} \Le PN^{*p}N^qN^{*q}N^pP= P|N|^{2r}P
\overset{\eqref{pukur}} = |\hat{T}|^{2(p+q)}.
   \end{align}
Since by \eqref{pukur}, $PN^{*q}N^p P = \hat
T^{*q} \hat T^p$, we see that $PN^{*q}N^p P$
is normal. In view of \eqref{fdsa} and
Lemma~\ref{rudec}, this implies that $P$
commutes with $N^{*q}N^p$, thus with
$|N^{*q}N^p|^2=|N|^{2(p+q)}$, and finally
with $|N|$.

In summary:\ $P$ commutes with $|N|$. This fact,
together with Lemma~\ref{nwer} shows that (i) holds.
Since, by assumption, $(p,q)\in \varXi$ and $p-q\neq
0$, the quantity $d$ is well defined. We now prove
that $P$ commutes with $U^d$. Using
Lemma~\ref{wurnik}, we get
   \begin{align} \label{kuket}
\text{$P$ commutes with $U^{m-n}$ for every $(m,n)\in
\varXi$.}
   \end{align}
Applying the well-ordering principle (see
Footnote~\ref{fitupi}), one can show that there exists
a finite nonempty subset $\varXi_0$ of $\varXi$ such
that
   \begin{align*}
\gcd\{m-n\colon (m,n)\in \varXi\} = \gcd\{m-n\colon
(m,n)\in \varXi_0\}.
   \end{align*}
By \cite[Theorem~4, p.\ 10]{Sh67}, for every
$(m,n)\in\varXi_0$, there exists $k_{m,n}\in \zbb$
such that
   \begin{align*}
d=\sum_{(m,n) \in \varXi_0} (m-n)k_{m,n}.
   \end{align*}
Combined with \eqref{kuket}, this implies
that $P$ commutes with $U^d$. Since, by
\eqref{pdu}, $U$ commutes with $|N|$, we
conclude that $P$ commutes with $N^{d}$. An
application of Lemma~\ref{npq} completes the
proof.
   \end{proof}
   \section{\label{Sec.6w}Proofs of Theorems~\ref{hypth}
and \ref{dsaqw}}
   We begin by characterizing the sets $G$ of
monomials in $t$ and $t^*$ that generate a commutative
unital $C^*$-algebra $\ascr=C^*(t)$. Set
   \begin{align} \label{saqur}
\mathfrak{G}_k = \big\{z \in \cbb\colon z^k = 1
\big\}, \quad k \in \natu,
   \end{align}
and define $\sgn(z)=\frac{z}{|z|}$ for $z \in \cbb
\setminus \{0\}$ and $\sgn(0)=0$.
   \begin{lemma}\label{winjkw}
Let $\ascr$ be a commutative unital $C^*$-algebra
generated by $t\in \ascr$ and let $G=\{t^{*m} t^n
\colon (m,n) \in \varXi\}$ with $\varXi
\subseteq\zbb_+^2$. Assume $\varXi_0 := \{(m,n)\in
\varXi\colon m\neq n\}$ is nonempty. Set
$X=\sigma(t)$. Then the following conditions are
equivalent{\em :}
   \begin{enumerate}
   \item[(i)] $\ascr=C^*(G)$,
   \item[(ii)] $\ascr=C^*(G_0)$, where
$G_0=\{t^{*m} t^n \colon (m,n) \in \varXi_0\}$,
   \item[(iii)] for any distinct points $z_1,z_2$
in $X$ with $|z_1|=|z_2|$ there exists $(m,n)\in
\varXi_0$ such that for every $u\in
\mathfrak{G}_{|n-m|}$, $\sgn(z_1) \neq u \sgn(z_2)$,
   \item[(iv)] the map $\varphi\colon
X \to Y$ is injective, where $Y:=
\smallcross_{(m,n)\in \varXi_0} Y_{m,n}$ with
$Y_{m,n}=\cbb$ and $\varphi(z) = \{\bar z^{m}
z^n\}_{(m,n)\in \varXi_0}$ for $z \in X$.
   \end{enumerate}
   \end{lemma}
   \begin{proof}
By \cite[Theorem~11.19]{Rud73}, there is no loss of
generality in assuming that $\ascr=C(X)$ and $t=\xi$,
where $\xi$ is as in \eqref{ksik}. For $(m,n)\in
\zbb_+^2$, define the function $\varphi_{m,n}\colon X
\to \cbb$ by $\varphi_{m,n}(z)=\bar z^{m} z^n$ for all
$z \in X$.

(i)$\Rightarrow$(iv) In view of the Stone-Weierstrass
theorem (see Theorem~\ref{wtysa}), it suffices to show
that if $G$ separates the points of $X$, then
$\varphi$ is injective. Assume $G$ separates the
points of $X$. Let $z_1, z_2$ be distinct points in
$X$. If $|z_1| \neq |z_2|$, then for every $(m,n)\in
\zbb_+^2 \setminus \{(0,0)\}$ (in particular, for
every $(m,n)\in \varXi_0$), $|\varphi_{m,n}(z_1)| \neq
|\varphi_{m,n}(z_2)|$ so $\varphi_{m,n}(z_1) \neq
\varphi_{m,n}(z_2)$. Therefore $\varphi(z_1) \neq
\varphi(z_2)$. Suppose now that $|z_1| = |z_2|$. Since
$G$ separates the points of $X$, there exists
$(m,n)\in \varXi$ such that $\varphi_{m,n}(z_1) \neq
\varphi_{m,n}(z_2)$. Then $(m,n)\in \varXi_0$ (because
otherwise $m=n \Ge 0$ and thus $\varphi_{m,n}(z_1) =
|z_1|^{2m} = |z_2|^{2m} = \varphi_{m,n}(z_2)$, a
contradiction). Hence $\varphi(z_1) \neq
\varphi(z_2)$.

(iv)$\Rightarrow$(iii) Let $z_1,z_2$ be distinct
points in $X$ such that $|z_1|=|z_2|$. By (iv), there
exists $(m,n)\in \varXi_0$ such that
$\varphi_{m,n}(z_1) \neq \varphi_{m,n}(z_2)$, or
equivalently that $\frac{\sgn(z_1)}{\sgn(z_2)} \notin
\mathfrak{G}_{|n-m|}$. This implies that there is no
$u\in \mathfrak{G}_{|n-m|}$ such that $\sgn(z_1) = u
\sgn(z_2)$.

(iii)$\Rightarrow$(ii) According to the
Stone-Weierstrass theorem, it is enough to show that
$G_0$ separates the points of $X$. In view of proof of
(i)$\Rightarrow$(iv), we can focus on the case of
distinct points $z_1, z_2$ in $X$ such that
$|z_1|=|z_2|$. This case, however, has been settled in
the paragraph above.

(ii)$\Rightarrow$(i) Trivial.
   \end{proof}
We now give a criterion for when a set of monomials in
$t$ and $t^*$ generates $\ascr$.
   \begin{lemma} \label{cxduns}
Let $\ascr$ be a commutative unital $C^*$-algebra
generated by $t\in \ascr$ and let $G=\{t^{*m}t^n\colon
(m,n)\in \varXi\}$ with nonempty $\varXi\subseteq
\zbb_+^2$. Set $X=\sigma(t)$. Then the following hold
$($see {\em Footnote~\ref{fitupi}}$)${\em :}
   \begin{enumerate}
   \item[(i)] if $\gcd\{m-n\colon (m,n)\in
\varXi\}=1$, then $\ascr=C^*(G)$,
   \item[(ii)] if $\ascr=C^*(G)$ and
   \begin{align} \label{fyrova}
\forall \; k\in \natu \setminus \{1\} \; \exists
z_1,z_2\in X\setminus \{0\} \colon z_1\neq z_2 \text{
and } \frac{z_1}{z_2} \in \mathfrak{G}_k,
   \end{align}
then $\varXi_0:=\{(m,n) \in \varXi\colon m\neq n\}
\neq \emptyset$ and $\gcd\{m-n\colon (m,n)\in
\varXi\}=1$.
   \end{enumerate}
   \end{lemma}
   \begin{proof}
(i) Suppose that $\gcd\{m-n\colon (m,n)\in
\varXi\}=1$. Then $\varXi_0\neq \emptyset$ (see
Footnote~\ref{fitupi}). By Lemma~\ref{winjkw}, it
suffices to show that condition (iii) of this lemma is
valid. Suppose, to the contrary, that there exist
distinct points $z_1,z_2$ in $X$ such that
$|z_1|=|z_2|$ and for every $(m,n)\in \varXi_0$,
$\frac{\sgn(z_1)}{\sgn(z_2)} \in
\mathfrak{G}_{|n-m|}$. Since, by
\cite[Lemma~3.1]{St-St17}, $\bigcap_{k \in J}
\mathfrak{G}_{k} = \mathfrak{G}_{\gcd(J)}$ for any
nonempty set $J \subseteq \natu$, and
$\gcd\{|m-n|\colon (m,n)\in \varXi_0\}=1$, we have
$\sgn(z_1)=\sgn(z_2)$, which contradicts the
assumption that $z_1\neq z_2$.

(ii) Assume that $\ascr=C^*(G)$ and \eqref{fyrova}
holds. As in the proof of Lemma~\ref{winjkw}, we can
assume that $\ascr=C(X)$ and $t=\xi$. Suppose,
contrary to our claim, that $\varXi_0 = \emptyset$. By
\eqref{fyrova}, there exist distinct $z_1,z_2\in
X\setminus \{0\}$ such that $\frac{z_1}{z_2} \in
\mathfrak{G}_2$. Then $|z_1|=|z_2|$ and thus
$\varphi_{m,n}(z_1) = \varphi_{m,n}(z_2)$ for all
$(m,n)\in \varXi$. Hence, by Theorem~\ref{wtysa},
$\ascr\neq C^*(G)$, a contradiction. Since
$\varXi_0\neq \emptyset$, we can define
$d:=\gcd\{m-n\colon (m,n)\in \varXi\} \in \natu$. It
remains to show that $d=1$. Suppose, to the contrary,
that $d \Ge 2$. Because of \eqref{fyrova}, there exist
distinct $z_1,z_2\in X\setminus \{0\}$ such that
$w:=\frac{z_1}{z_2} \in \mathfrak{G}_d$. By
Lemma~\ref{winjkw}(iv), there exists $(m,n)\in
\varXi_0$ such that $\varphi_{m,n}(z_1) \neq
\varphi_{m,n}(z_2)$. Hence $|w|=1$, $w^d=1$ and
$w^{n-m} \neq 1$. On the other hand, there exists
$l\in \zbb \setminus \{0\}$ such that $n-m=ld$,~so
$1=(w^{d})^l=w^{n-m}\neq 1$, a contradiction.
   \end{proof}
   \begin{rem}
Regarding Lemma~\ref{cxduns}, note that
condition~\eqref{fyrova} implies that $X$ is infinite.
To prove this, observe that condition \eqref{fyrova}
is equivalent to
   \begin{align} \label{dflrea}
\forall \; k\in \natu \setminus \{1\} \; \exists w\in
\mathfrak{G}_k\setminus \{1\}\; \exists z\in
X\setminus \{0\} \colon wz\in X.
   \end{align}
Suppose, to the contrary, that $X$ is finite. Then, by
\eqref{dflrea}, there exists $z_0 \in X \setminus
\{0\}$ such that the set $\{p\in \mathbb{P}\colon
\exists w\in \mathfrak{G}_p\setminus \{1\} \colon wz_0
\in X\}$ is infinite, where $\mathbb{P}$ stands for
the set of prime numbers. Using the fact that the sets
$\{\mathfrak{G}_p\setminus \{1\}\}_{p\in \mathbb{P}}$
are pairwise disjoint, we conclude that $X$ is
infinite, which is a contradiction. \hfill
$\diamondsuit$
   \end{rem}
   \begin{proof}[Proof of Theorem~\ref{hypth}]
Assume that $G$ generates $\ascr$. Since the
$C^*$-algebra $\ascr$ is commutative, $t$ is normal.
Replacing $t$ by $t^*$ if necessary, we can assume
that $p<q$. In view of Theorem~\ref{murzeqiv}, it
remains to show that condition (iii) of this theorem
is valid. Let $\pi\colon \ascr\to \ogr(\hh)$ be a
representation and $\varPhi\colon \ascr \to \ogr(\hh)$
be a UCP map such that
   \begin{align} \label{wuir}
\pi(g) = \varPhi(g), \quad g\in G.
   \end{align}
It follows from the Stinespring dilation theorem (see
\cite[Theorem~1]{Sti55}) that there exist a Hilbert
space $\kk\supseteq \hh$ and a representation
$\rho\colon \ascr \to \ogr(\kk)$ such that
   \begin{align} \label{viro}
\varPhi(a)= P\rho(a)|_{\hh}, \quad a \in
\ascr,
   \end{align}
where $P\in \ogr(\kk)$ is the orthogonal projection of
$\kk$ onto $\hh$. Set $T=\pi(t)$ and $N=\rho(t)$. Then
$T$ and $N$ are normal operators and by \eqref{wuir}
and \eqref{viro}, $T^{*m}T^n = PN^{*m}N^n|_{\hh}$ for
all $(m,n) \in \varXi$. Since $\rho$ is a
representation of $\ascr$, we deduce that $\tilde
X:=\sigma(N) \subseteq \sigma(t)$. Set $\tilde Y=
\smallcross_{(m,n)\in \varXi_0} \tilde Y_{m,n}$ with
$\tilde Y_{m,n}=\{\bar z^{m} z^n\colon z\in \tilde
X\}$. Define the maps $\phi_{m,n}\colon \tilde X \to
\tilde Y_{m,n}$ ($(m,n)\in \varXi_0$) and $\phi\colon
\tilde X \to \tilde Y$ by $\phi_{m,n}(z)=\bar z^{m}
z^n$ and $\phi(z) = \{\phi_{m,n}(z)\}_{(m,n)\in
\varXi_0}$ for $z \in \tilde X$, where
$\varXi_0:=\{(m,n) \in \varXi\colon m\neq n\}$. Note
that for every $(m,n)\in \varXi_0$, $\sigma(N^{*m}N^n)
= \phi_{m,n}(\tilde X)=\tilde Y_{m,n}$ (see \cite[(14)
on p.\ 158]{Bir-Sol87}) and the spectral measure
$E_{m,n}\colon \borel(\tilde Y_{m,n}) \to \ogr(\hh)$
of the normal operator $N^{*m}N^n$ takes the form
$E_{m,n}=E_N \circ \phi_{m,n}^{-1}$, where $E_N$
stands for the spectral measure of $N$ (see
\cite[Theorem~6.6.4]{Bir-Sol87}). Since the spectral
measures $\{E_{m,n}\}_{(m,n)\in \varXi_0}$ commute,
there exists the product spectral measure $E$ of
$\{E_{m,n}\}_{(m,n)\in \varXi_0}$ defined on
$\borel(\tilde Y)$ (see, e.g.,
\cite[Proposition~4]{Sto87}). Let $\varOmega$ be any
finite nonempty subset of $\varXi_0$ and let
$\varDelta_{m,n} \in \borel(\tilde Y_{m,n})$ for
$(m,n)\in \varOmega$. Set $\varDelta_{m,n}=\tilde
Y_{m,n}$ for $(m,n)\in \varXi_0 \setminus \varOmega$.
Then
   \allowdisplaybreaks
   \begin{align*}
& E_N \circ \phi^{-1} \bigg(\bigcross_{(m,n)\in
\varXi_0} \varDelta_{m,n}\bigg) = E_N
\bigg(\bigcap_{(m,n)\in \varOmega} \phi_{m,n}^{-1}
(\varDelta_{m,n})\bigg)
   \\
&= \prod_{(m,n)\in \varOmega} \Big(E_N \circ
\phi_{m,n}^{-1}\Big) (\varDelta_{m,n}) =
\prod_{(m,n)\in \varOmega} E_{m,n} (\varDelta_{m,n}) =
E \bigg(\bigcross_{(m,n)\in \varXi_0}
\varDelta_{m,n}\bigg).
   \end{align*}
By uniqueness of product spectral measure and the
regularity of $E_N \circ \phi^{-1}$ (because of the
continuity of $\phi$), this implies that $E_N \circ
\phi^{-1}=E$ (see \cite[Proposition~4]{Sto87}). It
follows from Theorem~\ref{main1}(ii) that $P$ commutes
with $N^{|m-n|}$ for every $(m,n)\in \varXi_0$. In
view of Lemma~\ref{npq}, $P$ commutes with $N^{*m}N^n$
for every $(m,n)\in \varXi_0$. Therefore, $P$ commutes
with $E_{m,n}$ for every $(m,n)\in \varXi_0$, which
implies that $P$ commutes with $E$ (see
\cite[Proposition~4]{Sto87}). As a consequence, $P$
commutes with $E_N \circ \phi^{-1}$. However,
according to Lemma~\ref{winjkw}, $\phi$ is injective,
so the map $\phi^{-1} \colon \borel(\tilde Y) \ni
\varDelta \mapsto \phi^{-1}(\varDelta)\in
\borel(\tilde X)$ is surjective (see
\cite[Proposition~16]{C-S-s19}). Hence, $P$ commutes
with $E_N$, or equivalently $P$ commutes with
$N=\rho(t)$. Using the continuity of $\rho$ and the
fact that $C^*(t)=\ascr$, we conclude that $P$
commutes with $\rho$, or equivalently that $\hh$
reduces $\rho$. Thus, by \eqref{viro}, $\varPhi$ is a
representation of $\ascr$. Since $ G$ generates
$\ascr$ and both representations $\pi$ and $\varPhi$
are continuous, we infer from \eqref{wuir} that
$\pi=\varPhi$. This shows that condition (iii) of
Theorem~\ref{murzeqiv} is valid, so $G$ is hyperrigid.
   \end{proof}
We now explain the placement of the points $(p,q),
(r,r) \in \varXi$ from Theorem~\ref{hypth}.
   \begin{rem} \label{ytsq}
If $(r,r)$ lies on the line $m+n=p+q$ or below it,
i.e., $1\Le 2r \Le p+q$, then by
\cite[Theorem~5.2]{P-S22-b}, applied to $(2r,p+q)$ in
place of $(p,q)$, we deduce that for any nonzero
Hilbert space $\hh$, there exists a selfadjoint
operator $T\in \ogr(\hh)$ and a compactly supported
semispectral measure $F\colon \borel(\cbb) \to
\ogr(\hh)$ with $\supp(F) \subseteq \rbb$, which is
not a spectral measure, such that \allowdisplaybreaks
   \begin{align*}
T^{*m}T^n= \int_{\cbb} \bar z^m z^n F(\D z), \quad
(m,n)\in \varXi,
   \end{align*}
where $\varXi=\{(p,q),(r,r)\}$. Then $(\hh,T,F)$
satisfies the if-clause of \eqref{ftint} for any
compact subset $X$ of $\rbb$ containing $\sigma(T)
\cup \supp(F)$, with $G=\{\bar \xi^m\xi^n\colon
(m,n)\in \varXi\}$, but it does not satisfy the
then-clause ($\xi$ is as in \eqref{ksik}). A quick
inspection of the proof of
\textup{(i)}$\Leftrightarrow$\textup{(ii)} in
Lemma~\ref{prejyd} shows that the argument does not
require the assumption $C^*(G)=C(X)$. Hence $G$ is not
hyperrigid in $C(X)$. If $p+q$ is odd, then $G$
generates $C(X)$ and is not hyperrigid in $C(X)$. The
case where $p+q$ is even is more subtle and is not
covered by the present $X$. Namely, if $X\subseteq
\mathbb{R}$ is compact, $1\Le 2r<p+q$, $p+q$ is even,
and $G$ generates $C(X)$, then $G$ is
hyperrigid~in~$C(X)$.~\hfill $\diamondsuit$
   \end{rem}
   \begin{proof}[Proof of Theorem~\ref{dsaqw}]
Case (i) is a consequence of Lemma~\ref{cxduns}(i) and
Theorem~\ref{hypth}. To prove case (ii), we first show
that condition (iii) of Lemma~\ref{winjkw} is
satisfied. So by this lemma $\ascr=C^*(G)$. Finally,
we can use Theorem~\ref{hypth}.
   \end{proof}
   \section{\label{Sec.6}Proof of Theorem~\ref{jydgunr}}
We begin with some characterizations of hyperrigidity
for singly generated commutative unital $C^*$-algebras
that will be used later.
   \begin{lemma} \label{prejyd}
Let $X$ be a nonempty compact subset of $\cbb$ and $G$
be a set of generators of $C(X)$. Then the following
conditions are equivalent{\em :}
   \begin{enumerate}
   \item[(i)] $G$ is hyperrigid,
   \item[(ii)] for every Hilbert space $\hh$, every
normal operator $T\in \ogr(\hh)$ with spectrum in $X$
and every semispectral measure $F\colon \borel(X) \to
\ogr(\hh)$,
   \begin{align} \label{ftint}
f(T)=\int_{X} f \D F \;\; \forall f\in G \implies
\text{$F$ is the spectral measure of $T$,}
   \end{align}
   \item[(iii)] for all Hilbert spaces $\hh$ and $\kk$   such
that $\hh\subseteq \kk$, and all normal operators
$T\in \ogr(\hh)$ and $N\in \ogr(\kk)$ with spectrum in
$X$,
   \begin{align*}
f(T) = Pf(N)|_{\hh} \;\: \forall f\in G \implies
PN=NP,
   \end{align*}
where $P$ stands for the orthogonal projection of
$\kk$ onto $\hh$.
   \end{enumerate}
Moreover, conditions {\em (i)-(iii)} are still
equivalent regardless of whether the Hilbert spaces
considered in either of them are separable or not,
   \end{lemma}
   \begin{proof}[Proof of Lemma~\ref{prejyd}]
Recall the description of representations of $C(X)$.
If $\pi\colon C(X) \to \ogr(\hh)$ is a representation
and $T:= \pi(\xi)$, where $\xi$ is as in \eqref{ksik},
then $T$ is normal and $\sigma(T) \subseteq
\sigma(\xi) = X$. Since $\pi(f)=f(T)$ for complex
polynomials $f$ in $\xi$ and $\bar \xi$, and both maps
$\pi$ and $f\mapsto f(T)$ are continuous on $C(X)$
(see \cite[Theorem~2.1.7]{Mur90} and
\cite[\S12.24]{Rud73}), the Stone--Weierstrass theorem
implies that $\pi(f)=f(T)$ for all $f\in C(X)$. The
converse implication is obvious.

(i)$\Leftrightarrow$(ii) follows from
Proposition~\ref{atrug}, Theorem~\ref{soptw}, and
Corollary~\ref{toeq-1} by identifying $X$ with
$\mathfrak{M}_{C(X)}$ and $\hat a$ with $a$ for $a\in
C(X)$.

(i)$\Rightarrow$(iii) Suppose $T\in \ogr(\hh)$ and
$N\in \ogr(\kk)$ are normal operators with spectrum in
$X$, $\hh\subseteq \kk$, and $f(T) = Pf(N)|_{\hh}$ for
all $f\in G$. Let $\pi\colon C(X)\to \ogr(\hh)$ be the
representation induced by $T$, and let $\varPhi\colon
C(X)\to \ogr(\hh)$ be the UCP map defined by
$\varPhi(f)= Pf(N)|_{\hh}$ for $f\in C(X)$. Since
$\pi|_G=\varPhi|_G$, Theorem~\ref{murzeqiv}~yields
   \begin{align} \label{fptry}
f(T)=Pf(N)|_{\hh}, \quad f\in C(X).
   \end{align}
Substituting $f=\xi$ and $f=\bar \xi \xi$ gives
$T=PN|_{\hh}$ and $T^*T=PN^*N|_{\hh}$. This yields
   \begin{align} \label{psptr}
(PN^*P)(PNP) = (T^*\oplus 0)(T\oplus 0) = T^*T\oplus 0
= PN^*NP,
   \end{align}
where $0$ denotes the zero operator on $\kk\ominus
\hh$. Since $N$ and $PNP$ are normal,
Lemma~\ref{rudec} gives $PN=NP$, which proves (iii).
Hence $\hh$ reduces $N$ and $T=N|_{\hh}$.

(iii)$\Rightarrow$(ii) Let $T\in \ogr(\hh)$ be a
normal operator with $\sigma(T)\subseteq X$, and let
$F\colon \borel(X) \to \ogr(\hh)$ be a semispectral
measure such~that
   \begin{align} \label{ftxf}
f(T)=\int_{X} f \D F, \quad f\in G.
   \end{align}
It follows from Naimark's dilation theorem (see
\cite[Theorem~6.4]{Ml78}) that there exist a Hilbert
space $\kk$ and a spectral measure $E\colon
\borel(X)\to \ogr(\kk)$ such that $\hh\subseteq \kk$,
   \begin{align} \label{fdpe}
F(\varDelta) = PE(\varDelta)|_{\hh} \; \forall
\varDelta \in \borel(X) \quad \& \quad \kk = \bigvee
\Big\{E(\varDelta)\hh\colon \varDelta \in
\borel(X)\Big\},
   \end{align}
where $P\in \ogr(\kk)$ is the orthogonal projection of
$\kk$ onto $\hh$. Note that if $\dim \hh \Le
\aleph_0$, then $\dim \kk \Le \aleph_0$ (cf.\
\cite[Proposition~1($\delta$)]{St-SzIII}). The
operator $N:=\int_{X} z E(\D z) $ is normal,
$\sigma(N) = \supp(E) \subseteq X$ (see
\cite[Theorem~6.6.2]{Bir-Sol87}), and, by \eqref{ftxf}
and \eqref{fdpe}, $f(T) = Pf(N)|_{\hh}$ for all $f\in
G$. Hence, by (iii), $\hh$ reduces $N$, and thus $\hh$
reduces $E$ (see \cite[Theorem~6.6.3]{Bir-Sol87}).
Therefore, by \eqref{fdpe}, $\kk=\hh$ and $E=F$. Since
both sides of \eqref{ftxf} define representations of
$C(X)$ agreeing on the set $G$ generating $C(X)$, we
obtain $f(T)=\int_{X} f \D F$ for all $f\in C(X)$, so
$F$ is the spectral measure~of~$T$.

A careful inspection of the above arguments shows that
conditions (i)-(iii) are equivalent for separable
Hilbert spaces. Applying the ``moreover'' part of
Theorem~\ref{murzeqiv} completes the proof.
   \end{proof}
   \begin{proof}[Proof of Theorem~\ref{jydgunr}]
(i)$\Rightarrow$(ii) Let $\hh$, $T_n$ and $T$ be as in
(ii) and let
   \begin{align} \label{wlimw}
\wwlim_{n\to \infty} f(T_n) = f(T), \quad f\in G.
   \end{align}
For $n\in \natu$, we denote by $F_n\colon \borel(\cbb)
\to \ogr(\hh)$ the semispectral measure of $T_n$ and
by $E_{T}$ the spectral measure of $T$. By
\eqref{fysk} and \eqref{wlimw}, we have
 \begin{equation*}
\wwlim_{n\to \infty}\int_{X} f \D F_n = \int_{X} f \D
E_{T},\quad f\in G.
   \end{equation*}
By identifying $X$ topologically and algebraically
with $\mathfrak{M}_{C(X)}$, and each $a \in C(X)$ with
$\hat a$, and applying Proposition~\ref{atrug} and
Theorems~\ref{soptw} and~\ref{murzeqiv}(v), we obtain
   \begin{equation*}
\sslim_{n\to \infty}\int_{X} f \D F_n = \int_{X} f \D
E_{T},\quad f\in C(X),
   \end{equation*}
which shows that (ii) holds.

(ii)$\Rightarrow$(iii) is obvious.

(iii)$\Rightarrow$(i) By Lemma~\ref{prejyd}, it
suffices to show that the present condition (iii)
implies condition (iii) of that lemma. Let $\hh$ and
$\kk$ be Hilbert spaces such that $\hh\subseteq \kk$
and let $T\in \ogr(\hh)$ and $N\in \ogr(\kk)$ be
normal operators with spectrum in $X$ such that
   \begin{align}  \label{ftpfs}
f(T) = Pf(N)|_{\hh}, \quad f\in G,
   \end{align}
where $P\in \ogr(\kk)$ is the orthogonal projection of
$\kk$ onto $\hh$. First, note that we can assume
without loss of generality that $\dim \kk \ominus
\hh\Ge \aleph_0$ (otherwise, we can consider the
normal operator $N \oplus \eta I_{\mathcal{L}}$, where
$\eta\in X$ and $\mathcal{L}$ is a Hilbert space of
dimension $\aleph_0$). Let $S\in \ogr(\kk \ominus
\hh)$ be a completely non-unitary isometry (which by
the Wold-von Neumann decomposition theorem is
unitarily equivalent to a unilateral shift of an
appropriate multiplicity). Set $V=I_{\hh} \oplus S \in
\ogr(\kk)$ and fix $\lambda\in X$. Define two
sequences $\{M_n\}_{n=1}^{\infty},
\{T_n\}_{n=1}^{\infty}\subseteq \ogr(\kk)$ by
   \begin{align} \label{mdnvuk}
M_n=V^nNV^{*n} \text{ and } T_n = M_n + \lambda
(I-V^nV^{*n}) \text{ for } n\in \natu.
   \end{align}
Note that $\ran(V^n)$ reduces $V^n B V^{*n}$ for any
$B\in \ogr(\kk)$. Thus $\ran(V^n)$ reduces $M_n$.
Knowing that $(I-V^nV^{*n})$ is the orthogonal
projection of $\kk$ onto $\nul(V^{*n})$, we get
   \begin{align} \label{orsim}
T_n = M_n|_{\ran(V^n)} \oplus \lambda
I_{\nul(V^{*n})}, \quad n\in \natu.
   \end{align}
Since $V^n$ is an isometry, the map $\varPsi_n\colon
\ogr(\kk) \ni B \mapsto V^nBV^{*n}|_{\ran(V^n)}\in
\ogr(\ran(V^n))$ is a representation of the
$C^*$-algebra $\ogr(\kk)$. Therefore, for any $n\in
\natu$, $M_n|_{\ran(V^n)}=\varPsi_n(N)$ is normal,
consequently $T_n$ is normal, and
   \begin{align}  \label{mnryn}
\sigma(M_n|_{\ran(V^n)}) = \sigma(\varPsi_n(N))
\subseteq \sigma(N), \quad n \in \natu.
   \end{align}
By \eqref{orsim} and \eqref{mnryn},
$\sigma(T_n)\subseteq X$ for all $n\in \natu$, and
   \allowdisplaybreaks
   \begin{align} \notag
f(T_n) & = f(\varPsi_n(N)) \oplus f(\lambda)
I_{\nul(V^{*n})} = \varPsi_n(f(N)) \oplus f(\lambda)
I_{\nul(V^{*n})}
   \\ \label{fumv}
&= V^n f(N) V^{*n} + f(\lambda) (I-V^n V^{*n}), \quad
f \in C(X).
   \end{align}
Since $S$ is a completely non-unitary isometry,
$\sslim_{n\to\infty}V^{*n} = P$. This yields
   \begin{align}     \label{wwcxx}
\wwlim_{n\to\infty} f(T_n) & \overset{\eqref{fumv}}= P
f(N)P + f(\lambda) (I-P) = P f(N)|_{\hh} \oplus
f(\lambda) I_{\kk\ominus \hh}
   \end{align}
for all $f\in C(X)$. Hence, using \eqref{ftpfs} and
the fact that $\sigma(T \oplus \lambda I_{\kk\ominus
\hh})\subseteq X$, we have
   \begin{align} \label{wwcyz}
\wwlim_{n\to\infty} f(T_n) = f(T)\oplus f(\lambda)
I_{\kk\ominus \hh} = f(T \oplus \lambda I_{\kk\ominus
\hh}), \quad f\in G.
   \end{align}
It follows from \eqref{wwcxx} and (iii) that
   \begin{align*}
Pf(N)|_{\hh} \oplus f(\lambda) I_{\kk\ominus \hh}
=\sslim_{n\to\infty} f(T_n) = f(T \oplus \lambda
I_{\kk\ominus \hh}) = f(T) \oplus f(\lambda)
I_{\kk\ominus \hh}
   \end{align*}
for all $f\in C(X)$. Therefore, $Pf(N)|_{\hh}=f(T)$
for all $f\in C(X)$. Arguing as in \eqref{fptry} and
\eqref{psptr}), we conclude that $PN=NP$.

A careful analysis of the above arguments shows that
conditions (i)-(iii) remain equivalent when the
Hilbert spaces involved are separable. (In the proof
of (iii)$\Rightarrow$(i) one should note that if $\hh$
and $\kk$ are separable, then, with no loss of
generality, we may assume that $\dim \kk \ominus
\hh=\aleph_0$ and that $S$ is a unilateral shift of
multiplicity~$1$.) Together with the ``moreover'' part
of Theorem~\ref{murzeqiv}, this completes the proof.
   \end{proof}
   \section{\label{Sec.7n}Proofs of Theorems~\ref{specth},
\ref{sotwotth} and \ref{wszb}}
   \begin{proof}[Proof of Theorem~\ref{specth}]
(i)$\Rightarrow$(ii) Apply the Stone-von Neumann
calculus.

(ii)$\Rightarrow$(i) Let $X$ be a compact subset of
$\cbb$ containing $\supp{F} \cup \supp{E_{T}}$, where
$E_{T}$ is the spectral measure of $T$. Applying the
Stone-Weierstrass theorem and Theorem~\ref{dsaqw}(i)
to $\ascr=C(X)$ and $t=\xi$ (see \eqref{ksik}) shows
that $G=\{\bar\xi^{m}\xi^n \colon (m,n)\in \varXi\}$
generates $C(X)$ and is hyperrigid in $C(X)$. By (ii),
$f(T)=\int_{X} f \D F$ for all $f\in G$. Using
Lemma~\ref{prejyd}(ii) completes the proof.
   \end{proof}
   \begin{proof}[Proof of Theorem~\ref{sotwotth}]
By \eqref{add-2}, $f(T_k)=T_k^{*m}T_k^n$ for
$f=\bar\xi^{m}\xi^{n}$ (see \eqref{ksik}) with $m,n
\in \zbb_+$. To show \eqref{dycnim}, first apply
Theorem~\ref{dsaqw} to $\ascr=C(X)$ and $t=\xi$
together with the Stone-Weierstrass theorem, and then
use implication (i)$\Rightarrow$(ii) of
Theorem~\ref{jydgunr} with
$G=\{\bar{\xi}^{m}\xi^{n}\colon (m,n)\in \varXi\}$.
   \end{proof}
Recall that $T\in \ogr(\hh)$ is {\em normaloid} if
$\|T^n\|=\|T\|^n$ for all $n\in \zbb_+$, and that
every subnormal operator is normaloid (see
\cite[Theorem~1 in \S 2.6.2]{Fur01}).
   \begin{lemma} \label{wydnu}
Let $T, T_k\in \bou(\hh)$ $($$k\in \natu$$)$ be
normaloid operators such that
   \begin{align} \label{fdex}
\wwlim_{k\to\infty} T_k^{*r}T_k^r = T^{*r}T^r
   \end{align}
for some $r\in \natu$. Then $R:=\sup_{k\Ge 1} \|T_k\|<
\infty$ and $\sigma(T_k) \cup \sigma(T) \subseteq
\overline{\mathbb{D}}_R$ for every $k\in \natu$, where
$\overline{\mathbb{D}}_R:=\{z\in \cbb\colon |z|\Le
R\}$.
   \end{lemma}
   \begin{proof}
It follows from the uniform boundedness
principle and \eqref{fdex} that $\sup_{k\Ge
1} \|T_k\| = \big(\sup_{k\Ge 1}
\|T_k^r\|\big)^{1/r} < \infty$, so
$R<\infty$. Applying \eqref{fdex} again,
we~get
   \begin{align*}
\|T^rh\|^2 = \is{T^{*r}T^rh}{h} = \lim_{k\to\infty}
\is{T_k^{*r}T_k^rh}{h} \Le R^{2r} \|h\|^2, \quad h\in
\hh,
   \end{align*}
which shows that $\|T\|= \|T^r\|^{1/r} \Le
R$. Hence $\sigma(T_k) \cup \sigma(T)
\subseteq \overline{\mathbb{D}}_R$ for all
$k\in \natu$.
   \end{proof}
   \begin{proof}[Proof of Theorem~ \ref{wszb}]
By Lemma~\ref{wydnu} and Theorem~\ref{sotwotth} with
$X=\overline{\mathbb{D}}_R$, we obtain \eqref{dycnim}.
Substituting $f=\bar\xi^{m}\xi^{n}$ into
\eqref{dycnim} and using \eqref{add-2}, we
obtain~\eqref{woso}.
   \end{proof}
The example below shows that the normal operators
appearing in Theorem~\ref{specth} cannot be replaced
by more general ones; even isometries are not
admissible.
   \begin{ex} \label{isuszn}
Let $V\in \ogr(\hh)$ be a non-unitary isometry. Then,
by \cite[Proposition~I.2.3]{SF70}, there exists a
unitary operator $U\in \ogr(\kk)$ extending $V$,
so\footnote{Condition \eqref{nagy} is a particular
case of the Sz.-Nagy dilation theorem (see
\cite[Theorem~I.4.2]{SF70}).}
   \begin{equation}\label{nagy}
V^n=PU^n|_\hh, \qquad n\in \zbb_+,
   \end{equation}
where $P\in \ogr(\kk)$ is the orthogonal projection of
$\kk$ onto $\hh$. Let $E\colon \borel(\mathbb{T}) \to
\ogr(\kk)$ be the spectral measure of $U$, and let
$F\colon \borel(\tbb) \to \ogr(\hh)$ be the
semispectral measure defined by \eqref{fdpe} with
$X=\tbb$. It follows from \eqref{nagy} that
   \begin{equation*}
V^n=\int_{\tbb} z^n F(\D z) \;\; \text{and} \;\;
V^{*n}=\int_{\tbb} \bar z^n F(\D z) \;\; \text{for
all} \;\; n\in\zbb_+.
   \end{equation*}
This, together with $V^*V=I$, implies that $V^{*m}V^n=
\int_{\tbb} \bar{z}^m z^n F(\D z)$ for all $m,n \in
\zbb_+$. Observe that $F$ is not spectral, because
otherwise
   \begin{equation*}
V^{*}V= \int_{\tbb} \bar{z}z F(\D z)=\int_{\tbb} z
F(\D z)\int_{\tbb} \bar{z} F(\D z)= VV^{*},
   \end{equation*}
which contradicts the assumption that $V$ is not
unitary.
   \hfill $\diamondsuit$
   \end{ex}
   \section{\label{Sec.9}Counterexamples via Choquet boundary}
We begin by presenting examples that discuss the
optimality of the characterizations of hyperrigidity
given in Section~\ref{Sec.m} (see the comments
immediately following Theorem~\ref{dsaqw}). For the
definition of $\mathfrak{G}_k$, see \eqref{saqur}.
   \begin{prop} \label{isnytos}
Let $n\in \natu\setminus \{1\}$ and $d_1, \ldots, d_n
\in \natu$ be such that $d_i$ does not divide $d_j$
whenever $i\neq j$ and $\gcd(d_1, \ldots, d_n)=1$. Let
$(p_1,q_1), \ldots, (p_n,q_n) \in \zbb_+^2$ be
distinct pairs such that $d_j = |q_j-p_j|$ for $1 \Le
j \Le n$. Suppose that $X$ is a compact subset of
$\cbb$ such that for some $\beta_1, \ldots, \beta_n
\in (0,1)$ with $\beta_1 + \ldots + \beta_n=1$,
   \begin{align} \label{porea}
1 \in X \text{ and } \bigcup_{j=1}^n r_j
\mathfrak{G}_{d_j} \subseteq X \text{ with }
r_j=\beta_j^{-\frac{1}{p_j+q_j}}.
   \end{align}
Let $\mcal=\lin(\{1\} \cup \{\bar{\xi}^{p_j}
\xi^{q_j}\colon 1 \Le j \Le n\})$. Then
$C^*(\mcal)=C(X)$ and $1\in X \setminus
\partial_{\mcal} X$. In particular, the set
$G:= \{\bar{\xi}^{p_j} \xi^{q_j}\colon 1 \Le j \Le
n\}$ is not hyperrigid in $C(X)$.
   \end{prop}
   \begin{proof}
Set $\alpha_j=\frac{\beta_j}{d_j}$ for $j=1,\ldots,n$.
Then $\{\alpha_j\}_{j=1}^n \subseteq (0,1)$ and
$\sum_{j=1}^n d_j \alpha_j=1$. For $j\in \{1, \ldots,
n\}$, let $\omega_j$ be a primitive $d_j$-root of
unity. Define the Borel probability measure $\mu$ on
$\cbb$ by $\mu = \sum_{j=1}^n \alpha_j
\sum_{k=0}^{d_j-1} \delta_{r_j\omega_j^k}$. According
to \eqref{porea}, $\supp(\mu)\subseteq X$. Then
   \begin{align*}
\int_{X} \bar z^m z^n \D \mu(z) = \sum_{j=1}^n
\alpha_j r_j^{m+n} \sum_{k=0}^{d_j-1}
\omega_j^{k(n-m)}, \quad m,n \in \zbb_+.
   \end{align*}
Let $\varepsilon_l = \pm 1$ be such that
$q_l-p_l=\varepsilon_l d_l$ for $l\in \{1,
\ldots,n\}$. Using \eqref{porea}, the assumption that
$d_i$ do not divide $d_j$ whenever $i\neq j$, and a
well-known property of sums of powers of roots of
unity, we obtain
   \begin{align*}
\int_{X} \bar z^{p_l} z^{q_l} \D \mu(z) = \sum_{j=1}^n
\alpha_j r_j^{p_l+q_l} \sum_{k=0}^{d_j-1}
\omega_j^{\varepsilon_l k d_l} = \sum_{j=1}^n \alpha_j
r_j^{p_l+q_l} d_j \delta_{j,l} = \alpha_l
r_l^{p_l+q_l} d_l =1
   \end{align*}
for all $l \in\{1, \ldots, n\}$. This implies that
$f(1)=\int_{X} f(z) \D \mu(z)$ for all $f\in \mcal$.
Since $\gcd(d_1, \ldots, d_n)=1$, we infer from
Lemma~\ref{cxduns}(i) that $C(X)=C^*(\mcal)$. By
Theorem~\ref{wtysa}, $\mcal$ separates the points of
$X$. Together with the fact that $1 \notin
\supp(\mu)$, this implies that $1\in X \setminus
\partial_{\mcal} X$ (see \cite[Proposition~6.2]{Phe02}).
It follows from \v{S}a\v{s}kin's theorem (see
\cite{Sas67}; see also \cite[Theorem, p.\ 48]{Phe02}),
Theorem~\ref{murzeqiv}(ii), and $X \neq
\partial_\mathcal{M}X$ that $\mcal$ is not hyperrigid in $C(X)$.
   \end{proof}
As shown in Theorem~\ref{main2} below, Choquet
boundary techniques, when combined with
Proposition~\ref{isnytos}, allow us to construct
sequences of normal operators whose behaviour is more
subtle than that obtained merely by negating
condition~(iii) in Theorem~\ref{jydgunr}. Before
proceeding, we need the following lemma.
   \begin{lemma} \label{wphnw}
If $T\in \ogr(\hh)$ and
$\{T_n\}_{n=1}^{\infty}\subseteq \ogr(\hh)$,
then
   \begin{align*}
\sslim_{n\to \infty} T_n = T \quad \iff
\quad \wwlim_{n\to \infty} T_n = T \;\; \&
\;\; \wwlim_{n\to \infty} T_n^*T_n = T^*T.
   \end{align*}
   \end{lemma}
   \begin{proof} The proof is a direct
consequence of the following identity
   \begin{align} \tag*{\qedhere}
\|T_n h- Th\|^2 = \is{T_n^*T_n h}{h} - 2
\mathrm{Re} \is{T_nh}{Th} + \is{T^*Th}{h},
\quad h\in \hh.
   \end{align}
   \end{proof}
   \begin{theorem} \label{main2}
Let $X$ be a nonempty compact subset of $\cbb$, and
let $\mathcal{M}$ be a subspace of $C(X)$ such that
$1\in \mathcal{M}$, $C^*(\mcal)=C(X)$, and $X \neq
\partial_\mathcal{M}X$. Then for every
infinite-dimensional Hilbert space $\hh$, there exist
normal operators $T\in \bou(\hh)$ and
$\{T_n\}_{n=1}^{\infty}\subseteq \bou(\hh)$, all with
spectrum in $X$, such that
   \begin{enumerate}
   \item[(i)] $\{T_n\}_{n=1}^{\infty}$ does not
converge to $T$ in the strong operator topology,
   \item[(ii)] $\{f(T_n)\}_{n=1}^{\infty}$
converges in the weak operator topology for all~$f\in
C(X)$,
   \item[(iii)] $\wwlim_{n\to\infty} f(T_n) = f(T)$
for all $f\in \mcal$,
   \item[(iv)]
for every finite or countably infinite set $G$ of
generators of $C(X)$ which is hyperrigid in $C(X)$,
there exists $f_0\in G \setminus \mcal$ such that
   \begin{align*}
\wwlim_{n\to\infty} f_0(T_n) \neq f_0(T).
   \end{align*}
   \end{enumerate}
   \end{theorem}
   \begin{proof}
Since $C^*(\mcal) = C(X)$, we infer from
Theorem~\ref{wtysa} that $\mcal$ separates the points
of $X$. Fix $\lambda_0 \in X
\setminus\partial_\mathcal{M}X$. As $1\in \mcal$, it
follows (cf.\ \cite[Proposition~6.2]{Phe02}) that
there exists a Borel probability measure $\mu$ on $X$
such that $\mu\neq \delta_{\lambda_0}$ and
   \begin{equation} \label{eqonm}
f(\lambda_0) = \int_X f\D \mu,\quad f\in
\mcal.
   \end{equation}

Let $T=\lambda_0 I_{\hh}$ and $F\colon \borel(X) \to
\ogr(\hh)$ be defined by
$F(\varDelta)=\mu(\varDelta)I_\hh$ for $\varDelta \in
\borel(X)$. Then $F$ is a semispectral measure. By
\eqref{eqonm}, we have $f(T) = \int_X f\D F$ for all
$f\in \mcal$. By Naimark's dilation theorem, there
exists a Hilbert space $\kk$ containing $\hh$ and a
spectral measure $E\colon \borel(X) \to \ogr(\kk)$
satisfying \eqref{fdpe}. Then the operator $N:=\int_X
z E(\D z)$ is a normal operator with $\sigma(N)
\subseteq X$ satisfying
   \begin{align} \label{dil}
f(T) = P f(N)|_{\hh}, \quad f\in \mcal.
   \end{align}
Note that $NP\neq PN$. Indeed, otherwise $\hh$ reduces
$E$, so by \eqref{fdpe}, $F$ is a spectral measure.
This implies that $\mu=\delta_{\lambda_1}$ for some
$\lambda_1 \in X$. Since $\mcal$ separates the points
of $X$, \eqref{eqonm} yields $\mu
=\delta_{\lambda_0}$, a contradiction.

We now show that $\dim \hh = \dim \kk$. Indeed, since
the $\sigma$-algebra $\borel(X)$ is generated by a
countable algebra $\uscr$, one can deduce from the
Carath\'{e}odory extension theorem (see, e.g.,
\cite[Appendix]{Sto87}) that
   \begin{align*}
\kk = \bigvee \Big\{E(\varDelta)\hh\colon
\varDelta \in \uscr\Big\} = \bigvee
\Big\{E(\varDelta)e_{\iota}\colon \varDelta
\in \uscr, \iota \in J\Big\},
   \end{align*}
where $\{e_{\iota}\}_{\iota\in J}$ is an
orthonormal basis of $\hh$. Therefore, since
$\dim \hh \Ge \aleph_0$, we get
   \begin{align*}
\dim \hh \Le \dim \kk \Le \card \uscr \cdot
\card J \Le \aleph_0 \cdot \dim \hh = \dim
\hh,
   \end{align*}
so $\dim \hh = \dim \kk$.

We now follow the proof of (iii)$\Rightarrow$(i) in
Theorem~\ref{jydgunr} with $\mcal$ in place of $G$. As
in that proof, we may assume without loss of
generality that $\dim \kk \ominus \hh\Ge \aleph_0$,
while still preserving \eqref{dil}, $\dim \hh = \dim
\kk$, and $NP\neq PN$ (however, the second identity in
\eqref{fdpe} may fail). Fix $\lambda\in X$ and define
the sequence $\{T_n\}_{n=1}^{\infty} \subseteq
\ogr(\kk)$ of normal operators with spectrum in $X$ as
in \eqref{mdnvuk}. Clearly, the operator $\tilde T:=T
\oplus \lambda I_{\kk\ominus \hh}$ is normal and
satisfies $\sigma(\tilde T) \subseteq X$. By
\eqref{wwcxx} and \eqref{wwcyz}, we have
   \allowdisplaybreaks
   \begin{gather} \label{qgdwq}
\textit{$\{f(T_n)\}_{n=1}^{\infty}$ converges in the
weak operator topology for every $f\in C(X)$,}
   \\ \label{gdwq}
\wwlim_{n\to\infty} f(T_n) =
f(\tilde T), \quad f\in \mcal.
   \end{gather}
Let $G$ be any finite or countably infinite set of
generators of $C(X)$ which is hyperrigid in $C(X)$.
Since $NP\neq PN$, it follows from implication
(i)$\Rightarrow$(iii) of Lemma~\ref{prejyd} that there
exists $f_0\in G$ such that $f_0(T) \neq P
f_0(N)|_{\hh}$. This yields
   \begin{align} \label{nytqet}
\wwlim_{n\to\infty} f_0(T_n) \overset{\eqref{wwcxx}}=
P f_0(N)|_{\hh} \oplus f_0(\lambda) I_{\kk\ominus \hh}
\neq f_0(\tilde T).
   \end{align}

Summarizing, the normal operators $\tilde T, T_n\in
\ogr(\kk)$ ($n\in \natu$) have spectra in $X$ and
satisfy \eqref{qgdwq}-\eqref{nytqet}. In particular,
this implies that $f_0\in G \setminus \mcal$. We
claim~that
   \begin{align} \label{nocerw}
\textit{$\{T_n\}_{n=1}^{\infty}$ does not converge to
$\tilde T$ in the strong operator topology.}
   \end{align}
Indeed, applying Theorem~\ref{dsaqw}(i) to
$\varXi=\{(0,1), (1,1)\}$, we deduce that
$G=\{\xi,\bar \xi \xi\}$ generates $C(X)$ and is
hyperrigid in $C(X)$. Consequently, the corresponding
vector $f_0$ belongs to $\{\xi,\bar \xi \xi\}$. This,
together with \eqref{nytqet}, implies that
$\wwlim_{n\to\infty} T_n \neq \tilde T$ or
$\wwlim_{n\to\infty} T_n^*T_n \neq \tilde T^* \tilde
T$. Thus, by Lemma~\ref{wphnw}, \eqref{nocerw} holds.
Since \(\dim \hh=\dim \kk\), the Hilbert spaces
\(\hh\) and \(\kk\) are unitarily equivalent. Passing
to the unitary equivalents of the constructed
operators \(\{T_n\}_{n=1}^{\infty}\) and \(\tilde T\)
therefore completes~the~proof.
  \end{proof}
\noindent {\bf Acknowledgements} We thank the reviewer
for the thorough review and suggestions, which led to
a significant improvement in the structure and clarity
of the paper.
   \appendix
   \section{UCP maps on commutative unital  $C^*$-algebras}
We begin by stating a complex version of the
Stone-Weierstrass theorem, which is not easily found
in the literature (see \cite[Theorem~44.7]{Will70} for
the real case):
   \begin{theorem}[Stone-Weierstrass theorem]
\label{wtysa}
If $X$ is a compact Hausdorff space and $G$
is a nonempty subset of $C(X)$, then
$C^*(G)=C(X)$ if and only if $G$ separates
the points of $X$.
   \end{theorem}
   \begin{proof}
For the ``if'' part, see \cite[Theorem~2.40]{Doug72}.
If the ``only if'' part were false, there would exist
distinct points $x_1,x_2 \in X$ such that
$f(x_1)=f(x_2)$ for all $f \in G$. Since point
evaluations are characters, this would imply
$f(x_1)=f(x_2)$ for all $f \in C^*(G)$. By the
normality of $X$, Urysohn's lemma would give $f_0 \in
C(X)$ with $f_0(x_1)\neq f_0(x_2)$, so that $f_0
\notin C^*(G)$, a contradiction.
   \end{proof}
The following observation plays a crucial role in this
paper.
   \begin{prop}\label{atrug}
Let $\ascr$ be a separable commutative unital
$C^*$-algebra. Then $C(\mathfrak{M}_{\ascr})$ is
separable $($relative to the supremum norm$)$,
$\mathfrak{M}_{\ascr}$ is a separable metrizable and
compact Hausdorff space, and any Borel complex measure
on $\mathfrak{M}_{\ascr}$ is regular.
   \end{prop}
   \begin{proof}
Since, by the Gelfand-Naimark theorem, the Gelfand
transform $\ascr\ni a\mapsto \widehat{a} \in
C(\mathfrak{M}_{\ascr})$ is an isometric $*$-algebra
isomorphism, we see that $C(\mathfrak{M}_{\ascr})$ is
separable. By \cite[Theorem~2.7]{Gr-Ro16},
$\mathfrak{M}_{\ascr}$ is a metrizable compact
Hausdorff space. As a consequence, any nonempty open
subset of $\mathfrak{M}_{\ascr}$ is $\sigma$-compact.
Hence, by \cite[Theorem~2.18]{Rud87}, finite positive
Borel measures on $\mathfrak{M}_{\ascr}$ are
automatically regular. Using the Jordan decomposition
of real measures (see \cite[Section~6.6]{Rud87}), we
conclude that Borel complex measures on
$\mathfrak{M}_{\ascr}$ are regular. The metric space
$\mathfrak{M}_{\ascr}$, being compact is automatically
separable (see \cite[Theorem~ IV.\S41.VI.1]{Kur68}).
   \end{proof}
Now we state a Bochner-type theorem for UCP maps on
commutative unital $C^*$-algebras, which follows from
\cite[Proposition~4.5]{Paul02} via the
Gelfand--Naimark theorem. A semispectral measure
$F\colon \borel(X) \to \ogr(\hh)$ on a Hausdorff space
$X$ is called {\em regular} if the measure $\langle
F(\cdot)f,f\rangle$ is regular for all $f \in \hh$.
   \begin{theorem} \label{soptw}
Let $\ascr$ be a commutative unital $C^*$-algebra,
$\hh$ be a Hilbert space and $\pi,\varPhi\colon \ascr
\to \ogr(\hh)$ be arbitrary maps. Then the following
statements are valid{\em :}
   \begin{enumerate}
   \item[(i)] $\pi$ is a
representation of $\ascr$ if and only if
there exists a $($unique$)$ regular spectral
measure $E\colon
\borel(\mathfrak{M}_{\ascr}) \to \ogr(\hh)$
such that
   \begin{align} \label{pyia}
\pi(a) =\int_{\mathfrak{M}_{\ascr}}
\widehat{a} \, \D E,\quad a \in \ascr,
   \end{align}
   \item[(ii)] $\varPhi$ is a
UCP map if and only if there exists a
$($unique$)$ regular semispectral measure
$F\colon \borel(\mathfrak{M}_{\ascr}) \to
\ogr(\hh)$ such that
   \begin{align*}
\varPhi(a) =\int_{\mathfrak{M}_{\ascr}}
\widehat{a} \, \D F,\quad a \in \ascr.
   \end{align*}
   \end{enumerate}
Moreover, if $E\colon \borel(\mathfrak{M}_{\ascr}) \to
\ogr(\hh)$ is a regular semispectral measure
satisfying {\em \eqref{pyia}}, where $\pi$ is a
representation of $\ascr$, then $E$ is spectral.
   \end{theorem}
The spectral measure $E$ in Theorem~\ref{soptw}(i) is
called the {\em spectral measure} of $\pi$.
   \section{\label{App.B}Hyperrigidity revisited}
If $G$ is a nonempty subset of a unital $C^*$-algebra
$\ascr$ and $\pi\colon \ascr \to \ogr(\hh)$ is a
representation on a Hilbert space $\hh$, we say (see
\cite{Arv11}) that $\pi|_{ G}$ has the {\em unique
extension property} if the only UCP map $\varPhi\colon
\ascr \to \ogr(\hh)$ satisfying $\pi|_{G} =
\varPhi|_{G}$ is $\varPhi=\pi$ itself.
Theorem~\ref{murzeqiv} shows that Arveson's first four
characterizations of hyperrigidity and Kleski's (the
last) are equivalent in a more general context.
   \begin{theorem}[\mbox{\cite[Theorem~2.1]{Arv11} and
\cite[Proposition~2.2]{Kle14}}] \label{murzeqiv}
Suppose that $ G$ is a nonempty subset of a unital
$C^*$-algebra $\ascr$. Consider the following
conditions{\em :}
   \begin{enumerate}
   \item[(i)] $G$ is hyperrigid,
   \item[(ii)] for every Hilbert space $\hh$, every
representation $\pi\colon \ascr\to \ogr(\hh)$ and
every sequence $\varPhi_n\colon \ascr\to \ogr(\hh)$
$($$n\in \natu$$)$ of UCP maps,
   \begin{align*}
\lim_{n\to\infty}\|\varPhi_n(g)-\pi(g)\|=0 \;\;
\forall g\in G \implies
\lim_{n\to\infty}\|\varPhi_n(a)-\pi(a)\|=0 \;\;
\forall a\in \ascr,
   \end{align*}
   \item[(iii)] for every Hilbert space $\hh$ and every
representation $\pi\colon \ascr \to \ogr(\hh)$,
$\pi|_{ G}$ has the unique extension property,
   \item[(iv)] for every unital  $C^*$-algebra
$\bscr$, every unital $*$-homomorphism of
$C^*$-alge\-bras $\theta\colon \ascr \to \bscr$ and
every UCP map $\varPhi\colon \bscr \to \bscr$,
   \begin{align*}
\varPhi(x)=x \; \forall x \in \theta( G) \implies
\varPhi(x)=x \; \forall x \in \theta(\ascr),
   \end{align*}
   \item[(v)] for every Hilbert space $\hh$,
every representation $\pi\colon \ascr\to \ogr(\hh)$
and every sequence $\varPhi_n\colon \ascr\to
\ogr(\hh)$ $($$n\in \natu$$)$ of UCP maps,
   \begin{align*}
\wwlim_{n\to\infty}\varPhi_n(g)= \pi(g) \; \forall
g\in G \implies \sslim_{n\to\infty}\varPhi_n(a)=\pi(a)
\; \forall a\in \ascr.
   \end{align*}
   \end{enumerate}
Then conditions {\em (i)}-{\em (iv)} are equivalent.
Moreover, if $\ascr$ is separable, then conditions
{\em (i)}-{\em (v)} are equivalent, and what is more
they are still equivalent regardless of whether the
Hilbert spaces considered in either of them are
separable or not.
   \end{theorem}
   \begin{proof}
A careful analysis of the proof of
\cite[Theorem~2.1]{Arv11} shows that conditions
(i)-(iv) are equivalent without the assumption that
the Hilbert spaces appearing in them are separable,
and if $\ascr$ is separable, they are still equivalent
if the Hilbert spaces appearing in at least one of
them are separable (this also applies to
Definition~\ref{dyfnh}). If $\ascr$ is separable, then
conditions (iii) and (v) are equivalent by a standard
reduction to the cyclic case.
   \end{proof}
The corollary below is a direct consequence of
Theorems~\ref{soptw} and \ref{murzeqiv}.
   \begin{corollary}\label{toeq-1}
Let $G$ be a nonempty subset of a commutative unital
$C^*$-algebra $\ascr$. Then $G$ is hyperrigid if and
only if for any representation $\pi\colon \ascr\to
\ogr(\hh)$ on a Hilbert space $\hh$ and any regular
semispectral measure $F\colon
\borel(\mathfrak{M}_{\ascr}) \to \ogr(\hh)$,
   \begin{align*}
\pi(a) = \int_{\mathfrak{M}_{\ascr}} \widehat{a} \, \D
F \quad \forall a \in G \;\;\implies\;\; \textit{$F$
is the spectral measure of $\pi$.}
   \end{align*}
Moreover, if $\ascr$ is separable, these two
conditions are still equivalent regardless of whether
the Hilbert spaces considered in either of them are
separable or not.
   \end{corollary}
   \section{\label{AppC}Weak and strong closures of normal operators}
Theorem~\ref{jydgunr} is closely related to two
classical results in operator theory. The first is
Bishop's theorem, which says that if $\hh$ is any
Hilbert space, then the closure of the set
$\mathscr{N}$ of all normal operators on $\hh$ in the
strong operator topology coincides with the set of all
subnormal operators on $\hh$ (see
\cite[Theorem~II.1.17]{Con91}). The second is the
Conway-Hadwin theorem (see \cite[Theorem]{CoHad83}),
whose assertion (ii), when applied to the class
$\mathscr{S}=\mathscr{N}$, states that if $\dim
\hh=\aleph_0$, then an operator $T\in \ogr(\hh)$ is
the limit, in the weak operator topology, of a
sequence of normal operators on $\hh$ if and only if
$T$ is the compression of a normal operator $S\in
\ogr(\kk)$ to $\hh$ (i.e., $T=P S|_{\hh}$, where $\kk
\supseteq \hh$ and $P\in \ogr(\kk)$ is the orthogonal
projection of $\kk$ onto $\hh$). By Halmos's dilation
theorem, any operator $T$ on an (arbitrary) $\hh$ is
the compression of a normal operator to $\hh$. If
$\dim \hh \Ge \aleph_0$, then the closure of
$\mathscr{N}$ in the weak operator topology equals
$\ogr(\hh)$ (see \cite[Problem~224]{Hal82}). However,
if $2 \Le \dim \hh < \aleph_0$, then the sequential
closure of $\mathscr{N}$ equals $\mathscr{N}$, while
$\ogr(\hh)$ are compressions of normal operators.
   \bibliographystyle{amsalpha}
   
   \end{document}